\documentclass[12pt]{amsart}

\usepackage{amsmath,amssymb,latexsym, amsthm, amscd, mathrsfs, stmaryrd}
\usepackage[colorlinks=true,linkcolor=blue,citecolor=blue,linktocpage=true]{hyperref}

\usepackage{enumerate}
\allowdisplaybreaks
\numberwithin{equation}{section}

\newtheorem{proposition}{Proposition}[section]
\newtheorem{lemma}[proposition]{Lemma}
\newtheorem{corollary}[proposition]{Corollary}
\newtheorem{theorem}[proposition]{Theorem}
\newtheorem{definition}[proposition]{Definition}

\theoremstyle{remark}
\newtheorem{remark}[proposition]{Remark}

\newcommand{\U}{\mathbf U}
\newcommand{\Ui}{\U^\imath}
\newcommand{\tU}{\widetilde{\U}}
\newcommand{\tUi}{\widetilde{{\mathbf U}}^\imath}
\newcommand{\Sym}{\operatorname{Sym}\nolimits}
\newcommand{\qbinom}[2]{\begin{bmatrix} #1\\#2 \end{bmatrix} }
\newcommand{\TT}{\texttt{\rm T}}
\newcommand{\tUiD}{{}^{\mathrm{Dr}}\tUi}

\newcommand{\arxiv}[1]{\href{http://arxiv.org/abs/#1}{\tt arXiv:\nolinkurl{#1}}}

\newcommand{\Rmnum}[1]{\expandafter\@slowromancap\romannumeral #1@}
\def \TH{\Theta}
\def \g{\mathfrak{g}}

\def \hg{\widehat{\mathfrak{g}}}
\def \bH{{\mathbf H}}
\def \B{{\mathbf B }}
\def \K{\mathbb{K}}
\def \bC{\mathbb{C}}
\def \bQ{\mathbb{Q}}
\def \Z{\mathbb{Z}}
\def \I{\mathbb{I}}
\def \R{\mathbb{R}}
\def \bS{\mathbb{S}}

\def \cR{\mathcal {R}}

\def \bZ{\mathbb{Z}}

\def \bTH { \boldsymbol{\Theta}}
\def \bDel{ \boldsymbol{\Delta}}

\def \vs{\varsigma}
\def \bvs{\boldsymbol{\varsigma}}
\def \he{\mathrm{ht}}
\def \bvs{{\boldsymbol{\varsigma}}}

\begin{document}

\title[ A Drinfeld type presentation of affine $\imath$quantum groups, II]
{A Drinfeld type presentation of affine $\imath$quantum groups II: split BCFG type}

\author[Weinan Zhang]{Weinan Zhang}
\address{Department of Mathematics, University of Virginia, Charlottesville, VA 22904, USA}
\email{wz3nz@virginia.edu}




\begin{abstract}
Recently, Lu and Wang formulated a Drinfeld type presentation for $\imath$quantum group $\tUi$ arising from quantum symmetric pairs of split affine ADE type. In this paper, we generalize their results by establishing a current presentation for $\tUi$ of arbitrary split affine type. 
\end{abstract}

\maketitle
\setcounter{tocdepth}{1}
 \tableofcontents
 \section{Introduction}
\subsection{Background}
The affine quantum group, denoted by $\U$, admits two presentations: the Serre presentation introduced by Drinfeld-Jimbo, and the current presentation, also known as the Drinfeld presentation \cite{Dr88}. The isomorphism between these two presentations was stated by Drinfeld, and a detailed proof was supplied by Beck \cite{Be94} and Damiani \cite{Da12} \cite{Da15}. This current presentation has been shown to be crucial in the representation theory of affine quantum groups; see the survey paper \cite{CH10} for partial references. \par

Quantum symmetric pairs $(\U,\Ui_\bvs)$ were introduced by Letzter in \cite{Let99} for finite type and generalized to Kac-Moody type by Kolb \cite{Ko14}. The \emph{$\imath$quantum group} $\Ui_\bvs$ arising from quantum symmetric pairs is a coideal subalgebra of $\U$ associated with an involution on the underlying root system. 
The \textit{universal $\imath$quantum group} $\tUi$  \cite{LW19a} is a coideal subalgebra of Drinfeld double quantum group $\tU$, and $\Ui_{\bvs}$ can be recovered from $\tUi$ by a central reduction. The version $\tUi$ naturally arises in an $\imath$Hall algebra realization of $\imath$quantum groups \cite{LW19a}, and a braid group action on $\tUi$ is realized via the reflection functors in this approach \cite{LW21a}. \par

Both affine quantum groups and affine $\imath$quantum groups are closely related to the quantum integrable systems. The affine quantum groups, which are trigonometric solutions of the quantum Yang-Baxter equation (QYBE), arise in the quantum inverse scattering method \cite[\S 11,13]{Dr87}. The affine $\imath$quantum groups, as solutions for the reflection equation $(=$boundary QYBE), appear in the framework of quantum integrable systems with certain boundary conditions. The $\imath$quantum group of split affine rank one is known as the $q$-Onsager algebra, which plays a crucial role in the study of the XXZ open spin chain \cite{BB13}. The $\imath$quantum groups of higher ranks are known as the ``generalized $q$-Onsager algebras", whose applications in affine Toda field theory are developed \cite{BB10}. \par

According to \cite{BW18}, various algebraic, geometric, categorical results for quantum groups are expected to have corresponding analogues for $\imath$quantum groups. In particular, it is natural to ask whether there is a current presentation for $\Ui_\bvs$, analogous to the current presentation for $\U$. There are several early attempts to construct a current presentation for the $q$-Onsager algebra in \cite{BS10,BK20}.

A major progress toward this question is the recent work by Lu and Wang \cite{LW20b}, who formulated a Drinfeld type presentation for $\tUi$ of split affine ADE type. In the rank one case, Lu-Wang's Drinfeld type presentation was built on the construction of root vectors and relations by Baseilhac and Kolb \cite{BK20}. The braid group action on $\tUi$ plays an essential role in Lu and Wang's construction, just as for affine quantum groups \cite{Da93, Be94}.\par

\subsection{ Main results }

The goal of this paper is to generalize results in \cite{LW20b} to arbitrary split (untwisted) affine types, namely, to provide a Drinfeld type presentation for the split $\imath$quantum group $\tUi=\tUi(\hg)$ where $\g$ is any simple Lie algebra and $\hg$ is the corresponding untwisted affine Lie algebra. Generalizing from split ADE types to split BCFG types requires new ideas, which will be explained below; in the meantime, our proofs for several relations in the general setting also simplify Lu-Wang's original proofs for split ADE types.
\par
 Let us explain our approach in details. Let $(c_{ij})_{i,j\in\I_0}$ denote the Cartan matrix for $\g$. We start by recalling the Serre presentation of $\tUi$ in Definition~ \ref{def:tUi}, following \cite{LW19a} and \cite{LW20a}. Such a presentation is obtained by modifying the Serre presentation for $\Ui_{\bvs}$ given in \cite[Theorem 3.1]{CLW18} (built on earlier work by Kolb and Letzter for $\Ui_{\bvs}$). The braid group action on $\tUi$ (see Lemma \ref{lem:Ti}) arises from $\imath$Hall algebras \cite{CLW21, LW21a, LW21b}, and it recovers the braid group action established in \cite{KP11, BK20} for $\Ui_{\bvs}$ (who worked with specific parameters).
 \par
 We shall use similar definitions and notations \eqref{Bik}-\eqref{Thim} for root vectors in $\tUi$ (i.e. generators in the Drinfeld type presentation) as in \cite[(3.28)-(3.30)]{LW20b}; that is, we shall use $B_{i,k}$ for the real root vectors, $\acute{\TH}_{i,m}$ for the imaginary root vectors constructed in \cite{BK20} in the rank one case and adapted to our general case, and $\TH_{i,m}$ for the alternative imaginary root vectors originated from $\imath$Hall algebra approach \cite{LRW20}. We shall focus on $\TH_{i,m}$ instead of $\acute{\TH}_{i,m}$ throughout this paper.
\par
Let us explain how we formulate the relations \eqref{iDRG0}-\eqref{iDRG6'} for $B_{i,k},\TH_{i,m}$ which are used in a current presentation for $\tUi$. Relations \eqref{iDRG0}-\eqref{iDRG34} are natural generalizations of \cite[(3.33)-(3.37)]{LW20b} to the case when $c_{ij}$ is arbitrary.
\par
The remaining task is to obtain a suitable formulation of the Serre relations (in the current presentation). For $c_{ij}=-1$, there are two (equivalent) formulations of the Serre relations available: one is the general version formulated in \cite[(3.32),(3.38)]{LW20b}, and the other is the equal-index version \eqref{iDRG4'} (which is the special cases of the general version); the latter is much simpler than the general version and is obtained by applying degree shift automorphisms to the corresponding finite type Serre relation \eqref{eq:S2}. As for their generalizations to $c_{ij}<-1$, the general versions of Serre relations are going to be extremely complicated as Lu and Wang's formulation suggests. However, the equal-index versions \eqref{iDRG5'}-\eqref{iDRG6'} can be obtained relatively easily. Hence, we choose to use relations \eqref{iDRG4'}-\eqref{iDRG6'} in our current presentation. 
\par
 There are two supporting examples for the use of this equal-index version of Serre relations. For affine quantum groups, Damiani in \cite[Theorem 11.18]{Da12} showed that the Drinfeld presentation $ ^{\mathrm{Dr}}\U$ is equivalent to a reduced current presentation $ ^{\mathrm{Dr}}\U_{red}$ where Serre relations are replaced by the corresponding equal-index version. (See Proposition~ \ref{prop:Da}.) Moreover, Lu and Wang showed in \cite[\S 4.7-4.8]{LW20b} that their general Serre relations can be derived from other defining relations and the equal-index Serre relations \eqref{iDRG4'}. In other words, for $\tUi$ of split ADE type, if we replace the Serre relation in the current presentation formulated in \cite[Definition 3.10]{LW20b} by the equal-index version \eqref{iDRG4'} of itself, we obtain an equivalent presentation.
 \par
 Generalizing these phenomenons, we define an algebra $\tUiD_{red}$ in Definition~ \ref{def:Dpr} with defining relations \eqref{iDRG0}-\eqref{iDRG6'}. We shall show that $\tUiD_{red}$ is isomorphic to $\tUi$ in Theorem~ \ref{DprADE}. We call $\tUiD_{red}$ a reduced Drinfeld type presentation.
\par
In the proof of this isomorphism, most defining relations of $\tUiD_{red}$ are verified in $\tUi$ in a similar way as \cite[\S 4]{LW20b}. 
A major exception is the relation \eqref{iDRG1'} for $i\neq j$, whose original proof for $c_{ij}=-1$ in \cite[\S 4.7-4.8]{LW20b} is no longer effective when $c_{ij}<-1$, since it requires the help of the general Serre relation. In this paper, we provide a new inductive proof of the relation \eqref{iDRG1'} for $i\neq j$, which is based on a recursive formula \eqref{recur} uniformly established for arbitrary $c_{ij}$ in \S \ref{Verf1}.
\par
Surprisingly, the base case for the induction is much more challenging, and we establish it case by case depending on $c_{ij}$ in \S \ref{Verf2}. For $c_{ij}=-1$, the base case is verified using the finite type Serre relation. For $c_{ij}=-2$, the base case is derived from formulas for the braid group action. For $c_{ij}=-3$, it turns out we need both of the finite type Serre relation and formulas of the braid group action to prove the base case.
\par
Two techniques are widely used in our proof of \eqref{iDRG1'} and later in the proof of Serre relations, analogous to \cite{Da12}. One is the $q$-brackets, which allow us to write Serre relations and formulas of the braid group action in compact forms (e.g. \eqref{Se:brkt}, \eqref{Br:brkt1} etc.) and then deal with them efficiently. The other one is the degree shift automorphism $\TT_{\omega_i}$, coming from the braid group action, which sends $B_{i,k}$ to $B_{i,k-1}$ and fixes $B_{j,l},$ for all $j\neq i$. These degree shift automorphisms allow us to recover a general relation from a more basic version and thus minimize the required amount of work.
\par
It is still desirable to have general Serre relations, which we shall provide when $c_{ij}=-2$. We next explain our approach toward it.
\par
For $c_{ij}=-1$, the general Serre relation \eqref{iDRG5} is first formulated in \cite[(3.38), (5.6)]{LW20b}. We offer a more direct proof in terms of generating functions in \S \ref{SeADE}, compared with \cite[\S 4.7,4.8]{LW20b}. This new proof also offers a method to formulate general Serre relations which admits a natural generalization to cases $c_{ij}<-1$.
\par

For $c_{ij}=-2$, we generalize this method and formulate a general Serre relation \eqref{iDRG6} in terms of generating functions. Details for the proof of \eqref{iDRG6} is included in \S \ref{symfun}-\ref{genfor}. Such a formulation has several remarkable features. (Similar features also hold for $c_{ij}=-1$.)
    \begin{enumerate}
    \item Each component of its RHS is a finite sum, and the constant term is the same as the corresponding finite type Serre relation \eqref{eq:S3}. (see Remark \ref{rmk:com})
    \item It can be viewed as a deformation of the Serre relation in the original Drinfeld presentation. (see Remark \ref{rmk:spe})
    \end{enumerate}
Adding general Serre relations \eqref{iDRG5}-\eqref{iDRG6} to the reduced presentation, a (complete) Drinfeld type presentation for $\tUi$ of split affine type BCF is formulated and proved in Theorem \ref{DprBCF}.
 \par
For $c_{ij}=-3$, while it is still possible to formulate a version of the general Serre relation, computation becomes much more involved and we will skip it. For practical purposes such as developing the representation theory of $\tUi$, we do not need this.

\subsection{Organization}
This paper is organized as follows. In Section \ref{Prel}, we set up notations and review the basic theory of affine quantum groups and affine $\imath$quantum groups.
In Section \ref{ADE}, we formulate a current presentation for $\tUi$ of arbitrary split affine type in Definition \ref{def:Dpr} and Theorem \ref{DprADE}.
In Section \ref{BCF}, we establish a Drinfeld type presentation in terms of generating functions for $\tUi$ in Theorem \ref{DprBCF}.\par
In Section \ref{verf}, we verify the relation \eqref{iDRG1'} in the current presentation using an induction. We establish a recursive formula for the induction in Section \ref{Verf1} and check base cases in Section \ref{Verf2}.
In Section \ref{Serre}, we verify general Serre relations: the one for the $c_{ij}=-1$ case is verified in Section \ref{SeADE}, and the one for the $c_{ij}=-2$ case is verified in Section \ref{symfun}-\ref{genfor}.

\vspace{2mm}
\noindent {\bf Acknowledgement.}
The author would like to thank Ming Lu and his advisor Weiqiang Wang for sharing their work earlier and for many helpful discussions and advices. This work is partially supported by the GRA fellowship of Wang's NSF grant DMS-2001351.

\section{Preliminaries}\label{Prel}
\subsection{Affine Weyl groups}
Set $\I_0=\{1,\ldots,n\}$. Let $\g$ be a simple Lie algebra with Cartan matrix $ (c_{ij})_{i,j\in\I_0}$. Let $ d_i $ be relatively prime positive integers such that $(d_i c_{ij})_{i,j\in \I_0}$ is a symmetric matrix. 

Let $\cR_0$ denote the root system of $\g$. Fix a set of simple roots $\{\alpha_i|i\in \I_0\}$ for $\cR_0$ and denote the corresponding positive system by $\cR_0^+$. Let $ Q=\bigoplus_{i\in \I_0}\bZ \alpha_i$ be the root lattice of $\g$ and $P$ be the dual lattice of $Q$. The bilinear pairing between $P,Q$ is denoted by $\langle \cdot, \cdot \rangle: P\times Q \rightarrow \Z$. The lattice $P$ is known as the weight lattice of $\g$ and $P=\bigoplus_{i\in \I_0}\bZ \omega_i$, where $\omega_i$ are fundamental weights of $\g$ given by $ \langle \omega_i, \alpha_j \rangle =\delta_{i,j}$. We identify $Q$ as a sublattice of $P$ via $\langle \alpha_i,\alpha_j\rangle=d_i c_{ij}$. Let $\theta$ be the highest root for $\g$.
\par
Set $\I=\I_0\cup\{0\}$. Let $\widehat{\mathfrak{g}}$ be the untwisted affine Lie algebra associated to $\g$ with the affine Cartan matrix $ (c_{ij})_{i,j\in\I}$. Extend $Q$ to the affine root lattice $\widetilde{Q}:=Q\oplus \Z\alpha_0$. It is known that the element $\delta=\alpha_0+\theta\in \widetilde{Q}$ satisfies $\langle \alpha_i, \delta\rangle =0, \forall i\in \I$. The root system $\cR$ and the set of positive roots $\cR^+$ for $\widehat{\g}$ are defined to be
\begin{align}
\cR &=\{\pm (\beta + k \delta) \mid \beta \in \cR_0^+, k  \in \Z\}  \cup \{m \delta \mid m \in \Z\backslash \{0\} \},
     \label{eq:roots}  \\
\cR^+ &= \{k \delta +\beta \mid \beta \in \cR_0^+, k  \ge 0\}
\cup  \{k \delta -\beta \mid \beta \in \cR_0^+, k > 0\}
\cup \{m \delta \mid m \ge 1\}.
  \label{eq:roots+}
 \end{align}
Let $s_i$ be the reflection acting on $\widetilde{Q}$ by
$s_i(x)=x-\langle x, \alpha_i \rangle\alpha_i$ for $i \in \I$. The Weyl group $W_0$ of $\g$ and the affine Weyl group $W$ of $\hg$ are subgroups of $\mathrm{Aut}(\widetilde{Q})$ generated by $s_i,i\in \I_0$ and by $s_i,i\in \I$, respectively.
\par
The extended affine Weyl group $\widetilde{W}$ is the semi-direct product $W_0 \ltimes P$. It is known that $W \cong W_0 \ltimes Q$ and thus $W$ is identified with a subgroup of $\widetilde{W}$. For $\omega\in P$, write $\omega $ for the element $(1,\omega)\in \widetilde{W}.$ For $s\in W_0$, write $s$ for the element $(s,0)\in \widetilde{W}$.
\par
 There is a $\widetilde{W}$-action on $\widetilde{Q}$ extending the $W_0$-action on $\widetilde{Q}$ such that $\omega(\alpha_i)=\alpha_i-\langle \omega, \alpha_i \rangle \delta$ for $\omega\in P,i\in \I$.
We identify $P/Q$ with a finite group $\Omega$ of Dynkin diagram automorphism, and thus $\widetilde{W} \cong\Omega \ltimes W$. The length function $l$ on $W$ extends to $ \widetilde{W}$ by setting $l(\tau w)=l(w)$ for $\tau \in \Omega, w\in W$.

\subsection{Drinfeld presentation for affine quantum groups}

Let $v$ be the quantum parameter and $v_i=v^{d_i}$. Define, for $n\in \Z, a\in \bQ(v)$
\[
[n]_{a}=\frac{a^n-a^{-n}}{a-a^{-1}}, \qquad [n]_a!=[n]_a[n-1]_a\cdots [1]_a,\qquad \qbinom{n}{s}_a=\frac{[n]_a!}{[s]_a![n-s]_a!}.
\]
Write $[A,B]=AB-BA$ and $[A,B]_a=AB-aBA$. \par
Let $\U$ be the Drinfled-Jimbo quantum group associated to $\hg$ with Chevalley generators $\{E_i,F_i,K_i^{\pm1}| i\in \I\}$. Let $\U^-$ be the subalgebra of $\U$ generated by $F_i,i\in \I$.\par
It was formulated in \cite{Dr88} that $\U$ is isomorphic to $^{\mathrm{Dr}}\U$, where $^{\mathrm{Dr}}\U$ is the $\bQ(v)$-algebra generated by $x_{i k}^{\pm}$, $h_{i l}$, $K_i^{\pm 1}$, $C^{\pm \frac12}$, for $i\in\I_0$, $k\in\Z$, and $l\in\Z\backslash\{0\}$, subject to the following relations:
\begin{align}
&C^{\frac{1}{2}}, C^{- \frac12} \text{ are central,}\label{Dr1}
\\
[K_i,K_j] & =  [K_i,h_{j l}] =0, \quad K_i K_i^{-1} =C^{\frac12} C^{- \frac12} =1,
\\
[h_{ik},h_{jl}] &= \delta_{k, -l} \frac{[k c_{ij}]_{v_i}}{k} \frac{C^k -C^{-k}}{v_j -v_j^{-1}},\label{Dr3}
\\
K_ix_{jk}^{\pm} K_i^{-1} &=v_i^{\pm c_{ij}} x_{jk}^{\pm},
 \\
[h_{i k},x_{j l}^{\pm}] &=\pm\frac{[kc_{ij}]_{v_i}}{k} C^{\mp \frac{|k|}2} x_{j,k+l}^{\pm},\label{Dr5}
\\
[x_{i k}^+,x_{j l}^-] &=\delta_{ij} {(C^{\frac{k-l}2} K_i\psi_{i,k+l} - C^{\frac{l-k}2} K_i^{-1} \varphi_{i,k+l})}, 
 \\
x_{i,k+1}^{\pm} x_{j,l}^{\pm}-v_i^{\pm c_{ij}} x_{j,l}^{\pm} x_{i,k+1}^{\pm} &=v_i^{\pm c_{ij}} x_{i,k}^{\pm} x_{j,l+1}^{\pm}- x_{j,l+1}^{\pm} x_{i,k}^{\pm},\label{Dr7}
 \\
\Sym_{k_1,\dots,k_r}\sum_{t=0}^{r} (-1)^t \qbinom{r}{t}_{v_i}
& x_{i,k_1}^{\pm}\cdots
 x_{i,k_t}^{\pm} x_{j,l}^{\pm}  x_{i,k_t+1}^{\pm} \cdots x_{i,k_n}^{\pm} =0, \text{ for } r= 1-c_{ij}\; (i\neq j),\label{Dr8}
\end{align}
where
$\Sym_{k_1,\dots,k_r}$ denotes the symmetrization with respect to the indices $k_1,\dots,k_r$, $\psi_{i,k}$ and $\varphi_{i,k}$ are defined by the following functional equations:
\begin{align*}
1+ \sum_{m\geq 1} (v_i-v_i^{-1})\psi_{i,m}u^m &=  \exp\Big((v_i -v_i^{-1}) \sum_{m\ge 1}  h_{i,m}u^m\Big),
\\
1+ \sum_{m\geq1 } (v_i-v_i^{-1}) \varphi_{i, -m}u^{-m} & = \exp \Big((v_i -v_i^{-1}) \sum_{m\ge 1} h_{i,-m}u^{-m}\Big).
\end{align*}
We refer to \cite{Be94} for a proof of the isomorphism $\phi$: $^{\mathrm{Dr}}\U \to\U$.\par
In \cite[\S 11]{Da12}, the general Serre relation \eqref{Dr8} is proved to be redundant except the case $k_1=k_2=\cdots=k_{1-c_{ij}}$. Let $^{\mathrm{Dr}}\U_{red}$($red$ stands for the reduced presentation) denote the $\bQ(v)$-algebra generated by $x_{i k}^{\pm}$, $h_{i l}$, $K_i^{\pm 1}$, $C^{\pm \frac12}$ subject to relations \eqref{Dr1}-\eqref{Dr7} and
\begin{equation}
\sum_{t=0}^{r} (-1)^t \qbinom{r}{t}_{v_i}
 \big(x_{i,k }^{\pm}\big)^{r-t}  x_{j,l}^{\pm} \big( x_{i,k }^{\pm}\big)^t  =0, \text{ for } r= 1-c_{ij},k,l\in \Z\; (i\neq j).
 \end{equation}
\begin{proposition}[\text{\cite[Theorem 11.18]{Da12}}]\label{prop:Da}
$^{\mathrm{Dr}}\U$ is isomorphic to $^{\mathrm{Dr}}\U_{red} $ by sending generators $x_{i k}^{\pm}$, $h_{i l}$, $K_i^{\pm 1}$, $C^{\pm \frac12}$ to those with the same names.
\end{proposition}
Note Damiani's original result is stronger than the one stated in Proposition \ref{prop:Da}, since it also involves a reduction of relations \eqref{Dr3} and \eqref{Dr5}, but this version is sufficient for our purpose.\par
Compose the isomorphism in Proposition \ref{prop:Da} with $\phi$, we have an isomorphism
\begin{equation}\label{phired}
\phi_{red}:\,^{\mathrm{Dr}}\U_{red}\longrightarrow \U.
\end{equation}

\subsection{Universal $\imath$quantum groups of split affine type}
We recall the definition of the universal $\imath$quantum group of split affine type via its Serre presentation following \cite[\S 3.3]{LW20b}.
\begin{definition}\label{def:tUi}
The universal (split) affine $\imath$quantum group $\tUi:=\tUi(\hg)$ associated to $\hg$ is the $\bQ(v)$-algebra generated by $B_i,\K_i^{\pm 1}, i\in \I$, subject to
\begin{align}
\K_i\K_i^{-1} =\K_i^{-1}\K_i=1, & \quad \K_i  \text{ is central},
\\
B_iB_j -B_j B_i&=0, \quad \qquad\qquad\qquad\qquad\qquad \text{ if } c_{ij}=0,
 \label{eq:S1} \\
B_i^2 B_j -[2]_{v_i} B_i B_j B_i +B_j B_i^2 &= - v_i^{-1}  B_j \K_i,  \quad\qquad\qquad\qquad \text{ if }c_{ij}=-1,
 \label{eq:S2} \\
\sum_{r=0}^3 (-1)^r \qbinom{3}{r}_{v_i} B_i^{3-r} B_j B_i^{r} &= -v_i^{-1} [2]_{v_i}^2 (B_iB_j-B_jB_i) \K_i, \;   \text{ if }c_{ij}=-2,
 \label{eq:S3} \\
\label{eq:S4}
\sum_{s=0}^4(-1)^s
\qbinom{4}{s}_{v_i}
B_{i}^{4-s}B_{j} B_{i}^s &= -v_i^{-1}(1+[3]_{v_i}^2)( B_{j} B_{i}^2+ B_{i}^2 B_{j})\K_i \\\notag
&\quad+v_i^{-1}[4]_{v_i} (1+[2]_{v_i}^2) B_{i} B_{j} B_{i} \K_i \\\notag
& \quad-v_i^{-2}[3]^2_{v_i} B_{j} \K_i^2, \qquad\qquad\quad \text{ if } c_{ij}=-3.
\end{align}
\end{definition}
\begin{remark}\label{rmk:cred}
For any ${\bvs}=(\vs_i)_{i\in\I} \in (\bQ^\times)^\I $, an affine $\imath$quantum group $\Ui_{\bvs}$ with parameters is introduced in \cite{Ko14}, generalizing G. Letzter's work for finite type. $\Ui_{\bvs}$ admits a Serre presentation formulated in \cite[Theorem 7.1]{Ko14} and also in \cite[Theorem 3.1]{CLW18}.
\par
The presentation for $\tUi$ in Definition \ref{def:tUi} can be obtained by replacing the parameter $-v_i^2\vs_i$ in the Serre presentation of $\Ui_{\bvs}$ formulated in \cite[Theorem 3.1]{CLW18} by a central element $\K_i$ for $i\in \I$. (set $\tau=id$ there for split type) Hence, $\Ui_{\bvs}$ is related to $\tUi$ by a central reduction $\Ui_{\bvs}:= \tUi/(\K_i + v_i^2 \vs_i| i\in \I) $.
\end{remark}
\begin{remark} A Serre presentation for $\tUi$ is also formulated with generators $B_i,\widetilde{k}_i,i\in \I$ in \cite[Proposition 6.4]{LW19a} for finite ADE type and in \cite[Theorem 4.2]{LW20a} for symmetric Kac-Moody type. The central element $\K_i$ is related to $\widetilde{k}_i$ by $\K_i=-v_i^2\widetilde{k}_i$. We are following notations in \cite{LW20b} in this paper.
\end{remark}
\begin{remark}\label{deg}
$\tUi$ has a $\Z \I$-grading by setting
\[
\mathrm{wt}(B_i) = \alpha_i, \mathrm{wt}(\K_i)=2\alpha_i,\qquad i\in \I.
\]
We say that $B_i$ has weight $\alpha_i$.
\end{remark}

\begin{remark}\label{filt}
$\tUi$ has a natural filtered algebra structure by setting
\[
\tU^{\imath,m}= \bQ(v)\text{-span}\{B_{i_1} B_{i_2} \cdots B_{i_r}\K_\mu| \mu\in \mathbb{Z} \I, r\leq m , i_k\in\I\}.
\]
According to \cite{Let02,Ko14}, the associated graded algebra with respect to this filtration is
\begin{equation}
\mathrm{gr} \tUi \cong \U^- \otimes \bQ(v)[\K_i^\pm | i\in \I],
\qquad \overline{B_i}\mapsto F_i,  \quad
\overline{\K}_i \mapsto \K_i \; (i\in \I).
\end{equation}
\end{remark}

The following formulas for the braid group action on $\tUi$ of finite ADE type were obtained in \cite{LW21a}; its generalization to Kac-Moody types is conjectured in \cite[Conjecture 6.5]{CLW21} and proved in \cite{LW21b}.

\begin{lemma}[ \text{\cite[Lemma 5.1]{LW21a}, \cite[Conjecture 6.5]{CLW21}, \cite{LW21b}}]
\label{lem:Ti}
For $i\in \I$, there exists an automorphism $\TT_i$ of the $\bQ(v)$-algebra $\tUi$ such that
$\TT_i(\K_\mu) =\K_{s_i\mu}$, and
\[
\TT_i(B_j)= \begin{cases}
B_i \K_i^{-1},  &\text{ if }j=i,\\
B_j,  &\text{ if } c_{ij}=0, \\
B_jB_i-v_i B_iB_j,  & \text{ if }c_{ij}=-1, \\
 [2]_{v_i}^{-1} \big(B_jB_i^{2} -v_i[2]_{v_i} B_i B_jB_i +v^2 B_i^{2} B_j \big) + B_j\K_i,  & \text{ if }c_{ij}=-2,\\
 [3]_{v_i}^{-1}[2]_{v_i}^{-1} \big(B_jB_i^{3} -v_i[3]_{v_i} B_i B_jB_i^2 +v^2[3]_{v_i} B_i^{2} B_jB_i & \\
 -v_i^3 B_i^3 B_j + v_i^{-1}[B_j,B_i]_{v_i^3}\K_i \big)+[B_j,B_i]_{v_i}\K_i, & \text{ if }c_{ij}=-3.
\end{cases}
\]
for $\mu\in \Z\I$ and $j\in \I$.
Moreover,  $\TT_i$ $(i\in \I)$ satisfy the braid relations, i.e., $\TT_i \TT_j =\TT_j \TT_i$ if $c_{ij}=0$, and $\TT_i \TT_j \TT_i =\TT_j \TT_i \TT_j$ if $c_{ij}c_{ji}=1$, and $\TT_i \TT_j \TT_i \TT_j=\TT_j \TT_i \TT_j\TT_i$ if $c_{ij}c_{ji}=2$, and $\TT_i \TT_j \TT_i \TT_j \TT_i \TT_j=\TT_j \TT_i \TT_j\TT_i \TT_j \TT_i $ if $c_{ij}c_{ji}=3$.
\end{lemma}

Its inverse $\TT_i^{-1}$ is explicitly given by $\TT_i^{-1} (\K_\mu) =\K_{s_i\mu}$, and
\[
\TT_i^{-1} (B_j)=
\begin{cases}
B_i \K_i^{-1} ,  &\text{ if }j=i,\\
B_j,  &\text{ if } c_{ij}=0, \\
B_iB_j -v_i B_jB_i,  & \text{ if }c_{ij}=-1, \\
 {[}2]_{v_i}^{-1} \big( B_i^{2}B_j-v_i[2]_{v_i} B_i B_j B_i+v_i^2 B_j B_i^{2} \big) +B_j\K_i,  & \text{ if }c_{ij}=-2,\\
 [3]_{v_i}^{-1}[2]_{v_i}^{-1} \big(B_i^{3}B_j -v_i[3]_{v_i} B_i^2 B_j B_i +v^2[3]_{v_i} B_i B_j B_i^{2}&\\
 -v_i^3 B_j B_i^3 + v_i^{-1}[ B_i,B_j]_{v_i^3}\K_i \big)+[B_i,B_j]_{v_i}\K_i, & \text{ if }c_{ij}=-3.
\end{cases}
\]
\begin{remark}
For specific parameters $\bvs=(\vs_i)_{i\in \I},\vs_i=-v_i^{-2},i\in \I$, a braid group action on $\Ui_{\bvs}$ of split finite type is constructed in \cite[Theorem 3.3]{KP11}. By taking the central reduction in Remark \ref{rmk:cred}, $\TT_i$ descends to an automorphism of $\Ui_{\bvs}$, which recovers Kolb and Pellegrini's braid group action.\par

However, for general parameters $\bvs$, $\TT_i$ fails to become an automorphism of $\Ui_{\bvs}$ via the central reduction. (A quick way to see this: since $\TT_i(\K_i)=\K_i^{-1}$, if $\TT_i$ reduces to an automorphism on $\Ui_{\bvs}$, then the image of $\K_i$, as a scalar in $\Ui_{\bvs}$, must be $\pm 1$ and thus $\vs_i=\pm v_i^{-2}$.)\par
A natural generalization of Kolb and Pellegrini's braid group action to the split affine rank one case is formulated in \cite[\S 2]{BK20} for equal parameters $\bvs=(\vs_0,\vs_1),\vs_0=\vs_1$.
\end{remark}
For $w\in \widetilde{W}$ with a reduced expression $w =\sigma s_{i_1} \ldots s_{i_r},\sigma \in \Omega$, we define $\TT_w = \sigma \TT_{i_1} \ldots \TT_{i_r}$, where $\sigma$ acts on $\tUi$ by $\sigma(B_i) =B_{\sigma i}, \sigma(\K_i) =\K_{\sigma i}$, for all $i\in \I$. By Lemma~\ref{lem:Ti}, $\TT_w$ is independent of the choice of reduced expressions for $w$.
\par
The first property for this braid group action can be obtained by adapting \cite[\S2.7]{Lus89} to our setting.
\begin{lemma}
   \label{lem:Lus}
Let $x\in P$, $i, j \in \I_0$.
\begin{itemize}
\item[(a)]
If $s_i x=x s_i$, then $\TT_i \TT_x=\TT_x \TT_i$.
\item[(b)]
If $s_i x s_i= \alpha_i^{-1}x=\prod_{k\in \I_0} \omega_k^{a_k}$, then we have
$\TT_i^{-1}\TT_x\TT_i^{-1}=\prod_{k\in \I_0} \TT_{\omega_k}^{a_k}$, in particular, $\TT_i^{-1}\TT_{\omega_i}\TT_i^{-1}=\TT_{\omega_i}^{-1}\prod_{k\neq i}\TT_{\omega_k}^{-c_{ik}}$.
\item[(c)]
$\TT_{\omega_i} \TT_{\omega_j}  = \TT_{\omega_j} \TT_{\omega_i}.$
\end{itemize}
\end{lemma}
For $i \in \I$, just as in \cite[\S 3]{Be94}, let $\omega_i'= \omega_i  s_i$ and $\tUi_{[i]}$ be the subalgebra of $\tUi$ generated by
$B_i,  \TT_{\omega'_i}(B_i), \K_i, \K_{\delta-\alpha_i}.$ 
Since $l(\omega'_i)=l(\omega_i)-1$, we have
\begin{equation}
\TT_{\omega_i}=\TT_{\omega'_i} \TT_i.
\end{equation}

The following properties for this braid group action on $\tUi$ are natural generalizations of corresponding results formulated in \cite[\S 3.3]{LW20b}. 
\begin{lemma}[ \text{\cite[Lemma 3.5-3.6 and Proposition 3.9]{LW20b}} ]\label{lem:bra} Let $i\in \I$.
\begin{itemize}

\item[(a)] We have $\TT_w (B_i) = B_{w i}$, for any $w \in W$ such that $wi \in \I$.

\item[(b)] We have $\TT_{\omega_j}(x)=x$, for any $j\neq i$ and $x\in\tUi_{[i]}$.

\item[(c)] There exists a $\bQ(v)$-algebra isomorphism $\aleph_i: \tUi(\widehat{\mathfrak{sl}}_2) \rightarrow \tUi_{[i]}$, which sends $B_1 \mapsto B_i, B_0 \mapsto \TT_{\omega_i'} (B_i), \K_1 \mapsto \K_i, \K_0 \mapsto \K_\delta \K_i^{-1}$.
\end{itemize}
\end{lemma}


\section{Drinfeld type presentations for affine $\imath$quantum groups}\label{Dpr}
\subsection{A reduced Drinfeld type presentation for $\tUi$ of split affine type}\label{ADE}
New generators $B_{i,k},\TH_{i,m}$ are introduced in \cite[(3.28)-(3.30)]{LW20b} for $\tUi$ of split affine ADE type. We define elements $B_{i,k},\TH_{i,m}$ basically in the same way for $\tUi$ of arbitrary split affine type. \par
Define a sign function
\[
o(\cdot): \I \longrightarrow \{\pm 1\},
\]
such that $o(i) o(j)=-1$ whenever $c_{ij} <0$.\par
Define elements $B_{i,k},\acute{\Theta}_{i,m},\TH_{i,m}$ in $\tUi$ for $i\in \I_0 $, $k\in \Z$ and $m\ge 1$ by
\begin{align}
B_{i,k} &= o(i)^k \TT_{\omega_i}^{-k} (B_i),
  \label{Bik} \\
\acute{\Theta}_{i,m} &=  o(i)^m \Big(-B_{i,m-1} \TT_{\omega_i'} (B_i) +v_i^{2} \TT_{\omega_i'} (B_i) B_{i,m-1}
\label{Thim1} \\
& \qquad\qquad\qquad\qquad + (v_i^{2}-1)\sum_{p=0}^{m-2} B_{i,p} B_{i,m-p-2}  \K_{i}^{-1}\K_{\delta} \Big),
\notag \\
\TH_{i,m} &=\acute{\Theta}_{i,m} - \sum\limits_{a=1}^{\lfloor\frac{m-1}{2}\rfloor}(v_i^2-1) v_i^{-2a} \acute{\Theta}_{i,m-2a}\K_{a\delta} -\delta_{m,ev} v_i^{1-m} \K_{\frac{m}{2}\delta}.
\label{Thim}
\end{align}
In particular, $B_{i,0}=B_i$. $B_{i,k},\TH_{i,l}$ are homogeneous with respect to $\Z \I$-grading on $\tUi$ with weights
\[
\mathrm{wt} (B_{i,k})=\alpha_i + k\delta,\qquad \mathrm{wt}(\TH_{i,l})=l\delta.
\]
Set ${\Theta}_{i,0} =(v_i-v_i^{-1})^{-1},$ and ${\Theta}_{i,m} =0,$ for $m<0.$

With root vectors defined above, a Drinfeld type presentation for the affine $\imath$quantum group of split ADE type is introduced in \cite[\S 3.4]{LW20b}.
By replacing $v$ by $v_i$ and adding the equal-index version of Serre relations, a current presentation for $\tUi$ of arbitrary split affine type is given in Definition \ref{def:Dpr}. We call it a reduced Drinfeld type presentation for $\tUi$, since it is an $\imath$analogue of reduced Drinfeld presentation $\U_{red}$ for affine quantum groups.

\begin{definition}\label{def:Dpr}
Let $ \tUiD_{red} $ be the $\bQ(v)$-algebra generated by $\K_{i}^{\pm1}$, $C^{\pm1}$, $H_{i,m}$ and $B_{i,l}$, where  $i\in \I$, $m \in \Z_{\geq 1}$, $l\in\Z$, subject to the following relations, for $m,n \in \Z_{\geq1}$ and $k,l\in \Z$:
\begin{align}
& \K_i, C \text{ are central, }\quad
[H_{i,m},H_{j,n}]=0, \quad \K_i\K_i^{-1}=1, \;\; C C^{-1}=1,\label{iDRG0}
\\
&[H_{i,m},B_{j,l}]=\frac{[mc_{ij}]_{v_i}}{m} B_{j,l+m}-\frac{[mc_{ij}]_{v_i}}{m} B_{j,l-m}C^m,\label{iDRG1'}
\\
&[B_{i,k}, B_{j,l+1}]_{v_i^{-c_{ij}}}  -v_i^{-c_{ij}} [B_{i,k+1}, B_{j,l}]_{v_i^{c_{ij}}}=0, \text{ if }i\neq j,\label{iDRG2'}
 \\ %
&[B_{i,k}, B_{i,l+1}]_{v_i^{-2}}  -v_i^{-2} [B_{i,k+1}, B_{i,l}]_{v_i^{2}}
=v_i^{-2}\Theta_{i,l-k+1} C^k \K_i-v_i^{-4}\Theta_{i,l-k-1} C^{k+1} \K_i\label{iDRG3'}
\\
&\qquad\qquad\qquad\qquad\qquad\qquad\quad\quad\quad
  +v_i^{-2}\Theta_{i,k-l+1} C^l \K_i-v_i^{-4}\Theta_{i,k-l-1} C^{l+1} \K_i, \notag
\\\label{iDRG34}
&[B_{i,k} ,B_{j,l}]=0,\qquad   \text{ if }c_{ij}=0,
\\
&   \sum_{s=0}^2(-1)^s
\qbinom{2}{s}_{v_i}
B_{i,k}^{2-s}B_{j,l} B_{i,k}^s =-v_i^{-1}   B_{j,l} \K_i C^k,\qquad\text{ if } c_{ij}=-1,\label{iDRG4'}
\\
\label{iDRG5'}
&\sum_{s=0}^3(-1)^s
\qbinom{3}{s}_{v_i}
B_{i,k}^{3-s}B_{j,l} B_{i,k}^s =-v_i^{-1} [2]_{v_i}^2 (B_{i,k} B_{j,l}-B_{j,l} B_{i,k})\K_i C^k,\qquad\text{ if } c_{ij}=-2,
\\\label{iDRG6'}
&\sum_{s=0}^4(-1)^s
\qbinom{4}{s}_{v_i}
B_{i,k}^{4-s}B_{j,l} B_{i,k}^s = -v_i^{-1}(1+[3]_{v_i}^2)( B_{j,l} B_{i,k}^2+ B_{i,k}^2 B_{j,l})\K_i C^k\\\notag
&\qquad  +v_i^{-1}[4]_{v_i} (1+[2]_{v_i}^2) B_{i,k} B_{j,l} B_{i,k} \K_i C^k-v_i^{-2}[3]^2_{v_i} B_{j,l} \K_i^2 C^{2k} ,\qquad\text{ if } c_{ij}=-3,
\end{align}
where $\Theta_{i,m}$ are related to $H_{i,m}$ by the following functional equation:
\begin{align}\label{H-TH}
1+ \sum_{m\geq 1} (v_i-v_i^{-1})\Theta_{i,m} u^m  = \exp\Big( (v_i-v_i^{-1}) \sum_{m\geq 1} H_{i,m} u^m \Big).
\end{align}
\end{definition}

\begin{theorem}\label{DprADE}
There is a $\bQ(v)$-algebra isomorphism $\Phi_{red}: \tUiD_{red} \to \tUi$, which sends
\begin{align*}
B_{i,k} \mapsto B_{i,k},\quad H_{i,k} \mapsto H_{i,k},\quad \TH_{i,k}\mapsto \TH_{i,k},\quad \K_i \mapsto \K_i, \quad C \mapsto \K_\delta,
\end{align*}
for $i\in \I_0,k\in \Z, m\geq 1$.
\end{theorem}
\begin{proof}Most defining relations for $^{Dr}\tUi_{red}$ are verified in $\tUi$ in similar ways as \cite{LW20b}, except the relation \eqref{iDRG1'} for $i\neq j$. We postpone details of the proof of this relation to Section \ref{verf}.\par
Relations \eqref{iDRG0}-\eqref{iDRG1'} for $i=j$ and the relation \eqref{iDRG3'} follow from Lemma \ref{lem:bra}(c) and the rank one computation offered by Lu and Wang; see \cite[Theorem 2.16]{LW20b} for a summary.\par The relation \eqref{iDRG2'} is proved as follows. Since $d_ic_{ij}=d_j c_{ji},$ we have $v_i^{c_{ij}}=v_j^{c_{ji}}$. Then the LHS of the relation \eqref{iDRG2'} is symmetric with respect to $(i,k)$ and $(j,l)$, i.e., we have
\[
[B_{i,k}, B_{j,l+1}]_{v_i^{-c_{ij}}}  - v_i^{-c_{ij}} [B_{i,k+1}, B_{j,l}]_{v_i^{c_{ij}}}=
[B_{j,l}, B_{i,k+1}]_{v_j^{-c_{ji}}}  - v_j^{-c_{ji}} [B_{j,l+1}, B_{i,k}]_{v_j^{c_{j i}}}.
\]
Hence, it suffices to prove the relation \eqref{iDRG2'} for $c_{ij}=-1,0$. For these two cases, \eqref{iDRG2'} is verified in the same way as \cite[\S 4.2]{LW20b} with the help of Lemma \ref{lem:Lus} and Lemma \ref{lem:bra}(b).\par
The verification of the relation \eqref{iDRG1'} for $i\neq j$ is given in Section \ref{verf}, using other defining relations \eqref{iDRG2'}-\eqref{iDRG3'}; note that proofs for these two relations, as provided above, do not need the relation \eqref{iDRG1'}, and hence we did not run into a circular. \par
The relation \eqref{iDRG0} for $i\neq j$ is verified using \eqref{iDRG1'} and \eqref{iDRG3'} in the similar way as \cite[\S 4.5]{LW20b}. Relations \eqref{iDRG34}- \eqref{iDRG6'} are obtained by applying $\TT_{\omega_i}^{-k} \TT_{\omega_j}^{-l}$ to finite type Serre relations \eqref{eq:S1}-\eqref{eq:S4} respectively. Hence, $\Phi_{red}$ is a well defined homomorphism.\par
The surjectivity and injectivity of $\Phi_{red}$ follows by similar arguments in \cite[proof of Theorem 3.13]{LW20b}. (For surjectivity, one need to replace all $ ^{\mathrm{Dr}}\tUi$ and $ ^{\mathrm{Dr}}\U$ in their arguments by $\tUiD_{red}$ and $ ^{\mathrm{Dr}}\U_{red} $ respectively, and follow similar arguments there.)
\end{proof}

Define generating functions
\begin{align}\label{eq:Genfun}
\begin{cases}
\B_{i}(z) =\sum_{k\in\Z} B_{i,k}z^{k},
\\
 \bTH_{i}(z)  =1+ \sum_{m > 0}(v_i-v_i^{-1})\Theta_{i,m}z^{m},
\\
\bH_i(u)=\sum_{m\geq 1} H_{i,m} u^m,
\\
\bDel(z)=\sum_{k\in\Z}  C^k z^k.
\end{cases}
\end{align}
Then \eqref{iDRG1'} can be written in terms of generating functions as
\begin{align}\label{H-TH2}
&(v_i-v_i^{-1})[\frac{\partial}{\partial z}\bH_i(z),\B_j(w)]\\\notag
= &\left(\frac{1}{1-v_i^{c_{ij}}zw^{-1}}-\frac{1}{1-v_i^{-c_{ij}}zw^{-1}}-\frac{1}{1-v_i^{c_{ij}} z w}+\frac{1}{1-v_i^{-c_{ij}}zw}\right)\B_j(w),
\end{align}
and \eqref{H-TH} can be written as
\begin{equation}
\bTH_i(z)=\exp((v_i-v_i^{-1}) \bH_i(z)).
\end{equation}
Conjugate \eqref{H-TH2} by $\bTH_i(z)$. Since $(v_i-v_i^{-1})\bTH_i(z)\frac{\partial}{\partial z}\bH_i(z)=\frac{\partial}{\partial z}\bTH_i(z) $, we have
\begin{align*}
&\frac{\partial}{\partial z}\big(\bTH_i(z) \B_j(w) \bTH_i^{-1}(z)\big) \\
= &\left(\frac{1}{1-v_i^{c_{ij}}zw^{-1}}-\frac{1}{1-v_i^{-c_{ij}}zw^{-1}}-\frac{1}{1-v_i^{c_{ij}} z wC}+\frac{1}{1-v_i^{-c_{ij}}zwC}\right)\bTH_i(z)\B_j(w)\bTH_i(z)^{-1}.
\end{align*}
By integrating it with respect to $z$, we obtain the following equivalent formulation of \eqref{H-TH2}
\begin{equation}\label{H-TH3}
\bTH_i(z) \B_j(w)=\left(\frac{1 -v_i^{-c_{ij}}zw^{-1}}{1 -v_i^{c_{ij}}zw^{-1}} \cdot \frac{1 -v_i^{c_{ij}} zw C}{1 -v_i^{-c_{ij}}zw C}\right)\B_j(w)\bTH_i(z).
\end{equation}
Hence, \eqref{H-TH2} is equivalent to \eqref{H-TH3}. Relation \eqref{H-TH3} can be written component-wisely as the following relation,
\begin{equation}\label{iDRG1''}
[\Theta_{i,k},B_{j,l}]+[\Theta_{i,k-2},B_{j,l}]C=v_i^{c_{ij}}[\Theta_{i,k-1},B_{j,l+1}]_{v_i^{-2c_{ij}}}+v_i^{-c_{ij}}[\Theta_{i,k-1},B_{j,l-1}]_{v_i^{2c_{ij}}}C.
\end{equation}
Thus, \eqref{iDRG1'} is equivalent to \eqref{iDRG1''}. See also \cite[Proposition 2.8]{LW20b} for the rank one case, and \cite[Proposition 3.12]{LW20b} for the general case.

\begin{corollary}
There exists a $\bQ(v)$-algebra antiautomorphism $\Psi: \tUiD_{red}  \to \tUiD_{red}$ given by
\begin{align*}
B_{i,k} &\mapsto B_{i,-k}, \qquad H_{i,l} \mapsto C^{-l}H_{i,l},\qquad \Theta_{i,r} \mapsto C^{-r} \Theta_{i,r},\\
C &\mapsto C^{-1}, \qquad \K_i \mapsto \K_i,
\end{align*}
for $k\in \Z, r,l>0,i\in \I_0$.
\end{corollary}
\begin{proof}
Note that $\Psi^2=1$. It is straightforward to check that $\Psi$ preserves defining relations \eqref{iDRG1'}-\eqref{iDRG6'} of $\tUiD_{red}$.
 Hence, $\Psi$ is a well-defined antiautomorphism of $\tUiD_{red}$. 
\end{proof}
By computing weights in the sense of Remark \ref{deg}, we can regard $\Psi$ as the reflection $\alpha+k\delta \mapsto \alpha-k\delta, \alpha\in \cR_0, k\in \Z$ on the affine root system $\cR$. The effect of $\Psi$ can be written in the generating function format by
\begin{align*}
\Psi(\B_i(w))=\B_i(w^{-1}),\qquad \Psi\big(\bDel(wz)\bTH_i(z)\big)=\bDel\big((zw)^{-1}\big)\bTH_i(w^{-1}).
\end{align*}
Let $\Theta_i(s,r)$ be a short notation for the RHS of \eqref{iDRG3'}, i.e.,
\begin{align}\label{eq:Th1}
\Theta_i(s,r):=(-\Theta_{i,s-r+1} C^r  +v_i^{-2}\Theta_{i,s-r-1} C^{r+1}  -\Theta_{i,r-s+1} C^s  +v_i^{-2}\Theta_{i,r-s-1} C^{s+1} )\K_i.
\end{align}
Note that $\Theta_i(s,r)=\Theta_i(r,s)$. By a direct computation, we have
\begin{align}
\Psi(\Theta_i(s,r))
&=\Theta_i(-r-1,-s-1).
\end{align}

\subsection{A Drinfeld type presentation for $\tUi$ of split affine BCF type}\label{BCF}
In this section, we add the general version of Serre relations to the current presentation given in Definition \ref{def:Dpr}; then we provide a Drinfeld type presentation in terms of generating functions in Theorem \ref{DprBCF} for $\tUi(\hg)$, where $\g$ is allowed to be any finite type except $G_2$.\par
Recall generating functions defined in \eqref{eq:Genfun}. Define $\bS(w_1,w_2,w_3|z;i,j)$ to be the following expression
\begin{align}
\Sym_{w_1,w_2,w_3}\left\{\sum_{r=0}^3(-1)^{3-r}\qbinom{3}{r}_{v_i}\B_{i}(w_1)\cdots\B_{i}(w_{r})\B_{j}(z)\B_{i}(w_{r+1})\cdots\B_{i}(w_3)\right\}\label{eq:BCF},
\end{align}
and similarly define $\bS(w_1,w_2|z;i,j)$ to be the following expression
\begin{align}
\Sym_{w_1,w_2}\left\{\sum_{r=0}^2(-1)^{ r}\qbinom{2}{r}_{v_i}\B_{i}(w_1)\cdots\B_{i}(w_{r})\B_{j}(z)\B_{i}(w_{r+1})\cdots\B_{i}(w_2)\right\}.\label{eq:ADE}
\end{align}
Denote
\begin{align*}
\phi_i(w_1,w_2,w_3)&=\frac{v_i^{-2 }w_2^{2}w_3^{-1}-w_2}{1+w_2 w_1^{-1}+w_1 w_3^{-1}+ w_2^2 w_3^{-2}+ w_2^2 w_1^{-1} w_3^{-1}+w_1 w_2 w_3^{-2}-([3]^2_{v_i}-3) w_2 w_3^{-1} },\\
\psi_i(w_1,w_2,w_3)&=\frac{-1-(1-v_i^{-2})w_2 w_3^{-1}+v_i^{-2}w_2^2 w_3^{-2}- w_1 w_3^{-1}+v_i^{-2}w_1 w_2 w_3^{-1} }{1+w_2 w_1^{-1}+w_1 w_3^{-1}+ w_2^2 w_3^{-2}+ w_2^2 w_1^{-2} w_3^{-1}+w_1 w_2 w_3^{-2}-([3]^2_{v_i}-3) w_2 w_3^{-1}}.
\end{align*}
Details for the proof of the following general version of Serre relations \eqref{iDRG5} and \eqref{iDRG6} are given in Section \ref{Serre}.
\begin{theorem}\label{DprBCF}
The universal affine $\imath$quantum group $\tUi$ is isomorphic to the $\bQ(v)$-algebra $ \tUiD $ which is defined by generators $\K_{i}^{\pm1}$, $C^{\pm1}$, $\Theta_{i,m},B_{i,k}$ $(i\in \I_0$, $m\geq 1$, $k\in\Z)$, subject to the following defining relations, for $i, j \in \I_0$:
\begin{align}
&\K_i,C \text{ are central, }\qquad  \bTH_i(z) \bTH_j(w) =\bTH_j(w) \bTH_i(z),
\label{iDRG1}
\\
& \B_j(w)  \bTH_i(z)
 = \left(
  \frac{1 -v_i^{c_{ij}}zw^{-1}}{1 -v_i^{-c_{ij}}zw^{-1}} \cdot \frac{1 -v_i^{-c_{ij}}zw C}{1 -v_i^{c_{ij}} zw C}
  \right)
\bTH_i(z) \B_j(w), \label{iDRG2}
\\
&(v_i^{c_{ij}}z -w) \B_i(z) \B_j(w) +(v_i^{c_{ij}}w-z) \B_j(w) \B_i(z)=0, \qquad \text{ if }i\neq j,\label{iDRG3a}
\\
&(v_i^2z-w) \B_i(z) \B_i(w) +(v_i^{2}w-z) \B_i(w) \B_i(z)\label{iDRG3b}
\\\notag
=&v_i^{-2} \frac{\bDel(zw)}{v_i-v_i^{-1}} \big( (v_i^2z-w)\bTH_i(w) +(v_i^2w-z)\bTH_i(z) \big)\K_{i},
\\
\label{iDRG4}&\B_i(w)\B_j(z)-\B_j(z)\B_i(w)=0, \qquad\text{ if }c_{ij}=0,
\\
\label{iDRG5}&\bS(w_1,w_2|z;i,j)
 \\\notag
=& -\Sym_{w_1,w_2} \frac{\bDel(w_1w_2)}{v_i-v_i^{-1}}\frac{[2]_{v_i} z w_1^{-1} }{1 -v_i^{2}w_2w_1^{-1}}[\bTH_i(w_2),\B_j(z)]_{v_i^{-2}}\K_i\\\notag
& -\Sym_{w_1,w_2} \frac{\bDel(w_1w_2)}{v_i-v_i^{-1}}\frac{1 +w_2w_1^{-1}}{1 -v_i^{2}w_2w_1^{-1}}[\B_j(z),\bTH_i(w_2)]_{v_i^{-2}}\K_i,
 \\
&\text{ if }c_{ij}=-1,
\notag
\\
& \label{iDRG6}\bS(w_1,w_2,w_3|z ;i,j) \\\notag
=& v_i[2]_{v_i}[3]_{v_i}z^{-1}\Sym_{w_1,w_2,w_3} \frac{\bDel(w_2w_3)}{v_i-v_i^{-1}}  \phi_i(w_1,w_2,w_3) \big[\B_i(w_1),[\B_j(z),\bTH_i(w_2)]_{v_i^{-4}}\big]\K_i\\\notag
&-[3]_{v_i}z^{-1}\Sym_{w_1,w_2,w_3} \frac{\bDel(w_2w_3)}{v_i-v_i^{-1}} \phi_i(w_1,w_2,w_3) \big[[ \B_j(z),\B_i(w_1)]_{v_i^{-2}},\bTH_i(w_2)\big]\K_i\\\notag
&-v_i[2]_{v_i}\Sym_{w_1,w_2,w_3} \frac{\bDel(w_2w_3)}{v_i-v_i^{-1}} \psi_i(w_1,w_2,w_3) \big[[\bTH_i(w_2),\B_j(z)]_{v_i^{-4}}, \B_i(w_1)\big]\K_i\\\notag
&+\Sym_{w_1,w_2,w_3} \frac{\bDel(w_2w_3)}{v_i-v_i^{-1}} \psi_i(w_1,w_2,w_3)\big[\bTH_i(w_2),[\B_i(w_1), \B_j(z)]_{v_i^{-2}}\big]\K_i,
\\\notag
& \text{ if } c_{ij}=-2.
\end{align}
where $ \phi_i(w_1,w_2,w_3) ,\psi_i(w_1,w_2,w_3)$ are defined above.
\end{theorem}
\begin{proof}
By Theorem \ref{DprADE}, it suffices to show that $\tUiD_{red}$ is isomorphic to $\tUiD$. The componentwise version of relations \eqref{iDRG3a}-\eqref{iDRG4} are the same as relations \eqref{iDRG2'}-\eqref{iDRG34}, and the componentwise version of the relation \eqref{iDRG2} is the relation \eqref{iDRG1''}. One can find a proof for this in \cite[Theorem 5.1]{LW20b}. By a direct computation, \eqref{iDRG4'} is the $w_1^k w_2^k z^l$ component of \eqref{iDRG5}, and \eqref{iDRG5'} is the $w_1^k w_2^k w_3^k z^l$ component of \eqref{iDRG6}. Hence, the map $\Phi_{red}:\tUiD_{red}\longrightarrow \tUiD$ by sending generators $\K_{i}^{\pm1}$, $C^{\pm1}$, $\Theta_{i,m},B_{i,k}$ to those with same names is a well-defined homomorphism.
\par
We will show in Section \ref{Serre} that relations \eqref{iDRG5} and
\eqref{iDRG6} can be derived from defining relations of $ \tUiD_{red}$. Thus, the inverse of $\Phi_{red}$ constructed in the obvious way is well-defined, which implies $\Phi_{red}$ is an isomorphism.
\end{proof}
\begin{remark}
As pointed out in the proof, when $\g$ is of ADE type, this presentation is identical to the one in \cite[\S 3.4]{LW20b} and thus Theorem \ref{DprBCF} can be viewed as a generalization of their work.
\end{remark}
\begin{remark}
As originally formulated in \cite[(3.32)(3.38)(5.6)]{LW20b}, \eqref{iDRG5} can be written component-wisely as
\begin{equation}
\bS(k_1,k_2|l;i,j)=\R(k_1,k_2|l; i,j),
\end{equation}
where
\begin{align}
\label{eq:Skk}\bS(k_1,k_2|l;i,j)
&=\Sym_{k_1,k_2} \big( B_{i,k_1} B_{i,k_2} B_{j,l} -[2]_{v_i} B_{i,k_1} B_{j,l} B_{i,k_2} + B_{j,l} B_{i,k_1} B_{i,k_2}\big),
\\
 \label{eq:Rkk}\R(k_1,k_2|l; i,j)
&=\Sym_{k_1,k_2}\K_i  C^{k_1}
\Big(-\sum_{p\geq 0} v_i^{2p}  [2]_{v_i} [\Theta _{i,k_2-k_1-2p-1},B_{j,l-1}]_{v_i^{-2}}C^{p+1}
\\\notag
&\qquad\qquad -\sum_{p\geq 1} v_i^{2p-1}  [2]_{v_i} [B_{j,l},\Theta _{i,k_2-k_1-2p}]_{v_i^{-2}} C^{p}
 - [B_{j,l}, \Theta _{i,k_2-k_1}]_{v_i^{-2}} \Big).
\end{align}
\end{remark}

\begin{remark}\label{rmk:com}
One can obtain the componentwise formulas of \eqref{iDRG6} by expanding the denominators of $\phi_i(w_1,w_2,w_3)$ and $\psi_i(w_1,w_2,w_3)$. Note that, after rewriting $w_3^{-1}$ as $ w_2 C$ using $\Delta(w_2 w_3)$, denominators of $\phi_i(w_1,w_2,w_3),\psi_i(w_1,w_2,w_3)$ have the form $1+A$ such that $w_3$ and nonpositive powers of $w_2$ do not appear in $A$. Hence, once we expand the denominators, each component of the RHS will be a finite sum.
\par The constant component of \eqref{iDRG6} is the same as \eqref{eq:S3}.
The general componentwise formula of \eqref{iDRG6} is, however, too complicated to write down.
\end{remark}

\begin{remark}
Relations \eqref{iDRG1}-\eqref{iDRG4} are homogeneous by a direct observation on their componentwise formulas. Relations \eqref{iDRG5} and \eqref{iDRG6} are also homogeneous, since they can be derived from relations \eqref{iDRG1}-\eqref{iDRG4} in Section \ref{Serre}.
\end{remark}

\begin{remark}\label{rmk:spe}
Recall the filtration and $\tU^{\imath,m}$ in Remark \ref{deg}. For any $\beta=\sum_{i\in \I} n_i \alpha_i\in \cR^+$, define its height to be
\[
\he^+(\beta)=\sum_{i\in \I} n_i.
\]
Let $d=\he(\delta)$. By similar arguments in \cite[Proposition 4.4]{BK20},
\[
B_{i,k}\in \tU^{\imath, 1+k d}\setminus \tU^{\imath, k d},\quad \TH_{i,l} \in \tU^{\imath,ld}\setminus \tU^{\imath, ld-1},\quad H_{i,l} \in \tU^{\imath,ld}\setminus \tU^{\imath, ld-1} ,
\]
and the images of $B_{i,k},\TH_{i,l},H_{i,l}$ in $\mathrm{gr}\Ui$ are, up to a $\bQ(v)[\K_i^{\pm 1}]$ multiple, Drinfeld generators $x_{i,-k}^-,\varphi_{i,-l},h_{i,-l}$ of $ \U^- $ respectively for $k\geq 0,l > 0$. Since $B_{i,-k}=-\frac{1}{[2]_{v_i}}[H_{i,k},B_i]C^{-1}+B_{i,k}C^{-1}$ for $k>0$, we have
\[
B_{i,-k}\in  \tU^{\imath, 1+kd}.
\]\par
We claim the $w_1^{k_1} w_2^{k_2} w_3^{k_3} z^l$ component of \eqref{iDRG6} for $k_1, k_2, k_3, l \geq 0$ reduces to the Serre relation \eqref{Dr8} in $\mathrm{gr}\tUi\cong \U^-\otimes \bQ(v)[\K_i^\pm|i\in \I]$. Observe that this component has the form
\begin{align}
\label{iDRG6''}&\Sym_{k_1,k_2,k_3}\sum_{t=0}^{3} (-1)^t \qbinom{3}{t}_{v_i}
B_{i,k_1}\cdots
 B_{i,k_t} B_{j,l}  B_{i,k_t+1} \cdots B_{i,k_n}\\\notag
=
\Sym_{k_1,k_2,k_3}\bigg(&\sum (*)\TH_{i,k_2-s}B_{j,l'}B_{i,k_{1}+t}+\sum (*)\TH_{i,k_2-s}B_{i,k_{1}+t}B_{j,l'}+\sum (*)B_{j,l'}\TH_{i,k_2-s}B_{i,k_{1}+t}\\\notag
+&\sum (*)B_{i,k_{1}+t}\TH_{i,k_2-s}B_{j,l'}+\sum (*)B_{j,l'}B_{i,k_{1}+t}\TH_{i,k_2-s}+\sum (*)B_{i,k_{1}+t}B_{j,l'}\TH_{i,k_2 -s}\bigg).
\end{align}
where coefficients $(*)$ lie in $\bQ(v)[\K_i^\pm|i\in \I]$ and each sum ranges in $0\leq s \leq k_2, -s\leq t \leq s,l'\in\{l,l+1\}$. By a direct computation of heights, the RHS lies in $\tU^{\imath,3+(k_1+k_2+k_3+l)d}$, while the LHS lies in $\tU^{\imath, 4+(k_1+k_2+k_3+l)d}\setminus\tU^{\imath, 3+(k_1+k_2+k_3+l)d}$. Hence, the RHS of \eqref{iDRG6''} disappears in $\mathrm{gr}\tUi$, and thus the componentwise version of \eqref{iDRG6} reduces to the Serre relation \eqref{Dr8} in the original Drinfeld presentation.
\end{remark}

\section{Verification of the relation \eqref{iDRG1'}}\label{verf}

 In this section, we establish the relation \eqref{iDRG1'} for $i\neq j$ in $\Ui$ and complete the proof of Theorem \ref{DprADE}. \par
Recall that \eqref{iDRG1'} is equivalent to \eqref{iDRG1''}. Hence, it suffices to show that \eqref{iDRG1''} for $i\neq j$ holds in $\Ui$. Fix $i\neq j\in \I_0$ and denote
\begin{equation}\label{Yk}
Y_{k,l}=[\Theta_{i,k},B_{j,l}]+[\Theta_{i,k-2},B_{j,l}]C-v_i^{c_{ij}}[\Theta_{i,k-1},B_{j,l+1}]_{v_i^{-2c_{ij}}}-v_i^{-c_{ij}}[\Theta_{i,k-1},B_{j,l-1}]_{v_i^{2c_{ij}}}C.
\end{equation}
Since $\TH_{i,0}=\frac{1}{v_i-v_i^{-1}}$ and $\TH_{i,k}=0,\forall k<0$ by our convention, $Y_{k,l}=0$ if $k\leq 0$ and the relation \eqref{iDRG1''} is equivalent to $Y_{k,l}=0$.
\par
We will show that $Y_{k,l}=0$ for $k>0,l\in \Z$ in $\tUi$ in this section, in order to verify the relation \eqref{iDRG1''}. Other two defining relations \eqref{iDRG2'} and \eqref{iDRG3'} of $\tUiD_{red}$ are allowed to be used in this section, since their proof does not need \eqref{iDRG1'}.
\par

Recall some basic properties for $q$-brackets, which shall be used heavily in various computations in this section as well as remaining sections.

\begin{lemma}[\text{\cite[Remark 4.17]{Da12}, also \cite[Introduction]{Ji98}}]\label{lem:jac}Let $a,b,c\in \tUi,u,v,w\in \bC(q)\setminus\{0\}$. We have
\begin{enumerate}
\item $[a,b]_u=-u [b,a]_{u^{-1}}$,
\item $\big[[a,b]_u,b\big]_v=\big[[a,b]_v,b\big]_u$,
\item $\big[[a,b]_u,c\big]_v=\big[a,[b,c]_{v/w}\big]_{uw}-u\big[b,[a,c]_w\big]_{v/uw}.$
\end{enumerate}
\end{lemma}

\subsection{An induction on $k$}\label{Verf1}
Since the index $l$ of $Y_{k,l}$ can be shifted using $\TT_{\omega_j}$, it suffices to focus on $k$. We first establish an inductive formula on $k$, which relates $Y_{k,l}$ and $Y_{k+2,l}.$ Such an induction is partially inspired by Damiani's reduction for the relation \eqref{Dr5} affine quantum group \cite[Proposition 7.15]{Da12}.
\begin{proposition}\label{prop:recur}
Let $k>1$ or $k=0$. We have for $l\in \Z$,
\begin{equation}\label{recur}
Y_{k+2,l}=v_i^{-2}Y_{k,l}C.
\end{equation}
We also have $Y_{3,l}=(1-v_i^{-2})Y_{1,l}C$ for $l\in \Z$.
\end{proposition}
\begin{proof}
Write $Y_{k+2,l}\K_i-v_i^{-2}Y_{k,l}C\K_i=\Sigma_1+\Sigma_2+\Sigma_3+\Sigma_4$ for $k>1$ or $k=0$ where each summand $\Sigma_i$ is defined and rewritten as follows:
\begin{align}\label{verf1}
\Sigma_1:=&\big[\Theta_{i,k+2}-v_i^{-2}\Theta_{i,k}C , B_{j,l}\big]\K_i\\\notag
=&-\big[[B_{i,k+2},B_i]_{v_i^2} , B_{j,l}\big]-\big[[B_{i,1},B_{i,k+1}]_{v_i^2} , B_{j,l}\big],\\
\label{verf2}\Sigma_2:=&[\Theta_{i,k}-v_i^{-2}\Theta_{i,k-2}C , ,B_{j,l}] \K_i C\\\notag
=& -\big[[B_{i,k+1},B_{i,1}]_{v_i^2} , B_{j,l}\big]-\big[[B_{i,2},B_{i,k}]_{v_i^2} , B_{j,l}\big],
\\
\label{verf3}v_i^{-c_{ij}}\Sigma_3:=&-[\Theta_{i,k+1}-v_i^{-2}\Theta_{i,k-1}C,B_{j,l+1}]_{v_i^{-2c_{ij}}}\K_i\\\notag
=& \big[[B_{i,k+1},B_i]_{v_i^2} , B_{j,l+1}\big]_{v_i^{-2c_{ij}}}+\big[[B_{i,1},B_{i,k}]_{v_i^2} , B_{j,l+1}\big]_{v_i^{-2c_{ij}}}\\\notag
=& \big[B_{i,k+1},[B_i , B_{j,l+1}]_{v_i^{-c_{ij}}}\big]_{v_i^{2-c_{ij}}}-v_i^2 \big[B_{i},[B_{i,k+1} , B_{j,l+1}]_{v_i^{-c_{ij}}}\big]_{v_i^{-2-c_{ij}}}\\\notag
 & +\big[B_{i,1},[B_{i,k} , B_{j,l+1}]_{v_i^{-c_{ij}}}\big]_{v_i^{2-c_{ij}}}-v_i^2 \big[B_{i,k},[B_{i,1} , B_{j,l+1}]_{v_i^{-c_{ij}}}\big]_{v_i^{-2-c_{ij}}}\\\notag
=& - \big[B_{i,k+1},[ B_{j,l} ,B_{i,1}]_{v_i^{-c_{ij}}}\big]_{v_i^{2-c_{ij}}}+v_i^2 \big[B_{i},[B_{j,l} ,B_{i,k+2}]_{v_i^{-c_{ij}}}\big]_{v_i^{-2-c_{ij}}}\\\notag
 & -\big[B_{i,1},[ B_{j,l},B_{i,k+1}]_{v_i^{-c_{ij}}}\big]_{v_i^{2-c_{ij}}}+v_i^2 \big[B_{i,k},[ B_{j,l},B_{i,2}]_{v_i^{-c_{ij}}}\big]_{v_i^{-2-c_{ij}}},
\\
 \label{verf4}v_i^{c_{ij}}\Sigma_4:=&-[\Theta_{i,k+1}-v_i^{-2}\Theta_{i,k-1}C,B_{j,l-1}]_{v_i^{2c_{ij}}}C\K_i \\\notag
=& \big[[B_{i,k+2},B_{i,1}]_{v_i^2} , B_{j,l-1}\big]_{v_i^{2c_{ij}}}+\big[[B_{i,2},B_{i,k+1}]_{v_i^2} , B_{j,l-1}\big]_{v_i^{2c_{ij}}}\\\notag
=& \big[B_{i,k+2},[B_{i,1} , B_{j,l-1}]_{v_i^{c_{ij}}}\big]_{v_i^{2+c_{ij}}}-v_i^2 \big[B_{i,1},[B_{i,k+2} , B_{j,l-1}]_{v_i^{c_{ij}}}\big]_{v_i^{-2+c_{ij}}}\\\notag
 & +\big[B_{i,2},[B_{i,k+1} , B_{j,l-1}]_{v_i^{c_{ij}}}\big]_{v_i^{2+c_{ij}}}-v_i^2 \big[B_{i,k+1},[B_{i,2} , B_{j,l-1}]_{v_i^{c_{ij}}}\big]_{v_i^{-2+c_{ij}}}\\\notag
=& -\big[B_{i,k+2},[B_{j,l} ,B_{i}]_{v_i^{c_{ij}}}\big]_{v_i^{2+c_{ij}}}+v_i^2 \big[B_{i,1},[ B_{j,l}, B_{i,k+1}]_{v_i^{c_{ij}}}\big]_{v_i^{-2+c_{ij}}}\\\notag
 & -\big[B_{i,2},[B_{j,l} , B_{i,k+1}]_{v_i^{c_{ij}}}\big]_{v_i^{2+c_{ij}}}+v_i^2 \big[B_{i,k+1},[B_{j,l} , B_{i,1}]_{v_i^{c_{ij}}}\big]_{v_i^{-2+c_{ij}}},
\end{align}
where relation \eqref{iDRG3'} is used in the first equality in each of \eqref{verf1}-\eqref{verf4}, and relation \eqref{iDRG2'} is used in the last equality of \eqref{verf3}-\eqref{verf4}. Now, adding \eqref{verf1}-\eqref{verf4} together, we have
\begin{align*}
&Y_{k+2,l}\K_i-v_i^{-2}Y_{k,l}C\K_i=\Sigma_1+\Sigma_2+\Sigma_3+\Sigma_4\\
=&-\big[[B_{i,k+2},B_i]_{v_i^2} , B_{j,l}\big]-\big[[B_{i,1},B_{i,k+1}]_{v_i^2} , B_{j,l}\big]\\
 &-\big[[B_{i,k+1},B_{i,1}]_{v_i^2} , B_{j,l}\big]-\big[[B_{i,2},B_{i,k}]_{v_i^2} , B_{j,l}\big]\\
 &-v_i^{c_{ij}}\big[B_{i,k+1},[ B_{j,l} ,B_{i,1}]_{v_i^{-c_{ij}}}\big]_{v_i^{2-c_{ij}}}+v_i^{2+c_{ij}} \big[B_{i},[B_{j,l} ,B_{i,k+2}]_{v_i^{-c_{ij}}}\big]_{v_i^{-2-c_{ij}}}\\
 &-v_i^{c_{ij}}\big[B_{i,1},[ B_{j,l},B_{i,k+1}]_{v_i^{-c_{ij}}}\big]_{v_i^{2-c_{ij}}}+v_i^{2+c_{ij}} \big[B_{i,k},[ B_{j,l},B_{i,2}]_{v_i^{-c_{ij}}}\big]_{v_i^{-2-c_{ij}}}\\
 &-v_i^{-c_{ij}}\big[B_{i,k+2},[B_{j,l} ,B_{i}]_{v_i^{c_{ij}}}\big]_{v_i^{2+c_{ij}}}+v_i^{2-c_{ij}} \big[B_{i,1},[ B_{j,l}, B_{i,k+1}]_{v_i^{c_{ij}}}\big]_{v_i^{-2+c_{ij}}}\\
 &-v_i^{-c_{ij}}\big[B_{i,2},[B_{j,l} , B_{i,k+1}]_{v_i^{c_{ij}}}\big]_{v_i^{2+c_{ij}}}+v_i^{2-c_{ij}} \big[B_{i,k+1},[B_{j,l} , B_{i,1}]_{v_i^{c_{ij}}}\big]_{v_i^{-2+c_{ij}}}\\
=&0,
\end{align*}
where the last step follows by a direct computation using Lemma \ref{lem:jac}. Hence, $Y_{k+2,l}=v_i^{-2}Y_{k,l}C$ for $k>1$ or $k=0$. For $k=1$, using a similar method, we have $Y_{3,l}=(1-v_i^{-2})Y_{1,l}C$.
\end{proof}

\subsection{Base cases}\label{Verf2}
By Proposition \ref{prop:recur}, $Y_{2m,l}$ is a scalar multiple of $Y_{0,l}$ and $Y_{2m-1,l}$ is a scalar multiple of $Y_{1,l}$ for $m>0,l\in \Z$. Since $Y_{0,l}=0$ as discussed in the beginning of Section \ref{verf}, it remains to show that $Y_{1,l}=0$ for $l\in \Z$.
\par
We explain the underlying idea for proving the base case $Y_{1,l}=0$, since details in the proof are quite technical. By the definition \eqref{Yk}, we have
\begin{equation}\label{Y1}
Y_{1,l}=[\TH_{i,1},B_{j,l}]-[c_{ij}]_{v_i} B_{j,l+1} + [c_{ij}]_{v_i} B_{j,l-1}C.
\end{equation}
Since $\TH_{i,1}=-[B_{i,1},B_i]_{v_i^2}$ by \eqref{iDRG3'}, we replace $\TH_{i,1}$ in \eqref{Y1} by the $q$-brackets of real root vectors and we obtain
\begin{equation}\label{Y1'}
Y_{1,l}=-\big[[B_{i,1},B_i]_{v_i^2},B_{j,l}\big]-[c_{ij}]_{v_i} B_{j,l+1} + [c_{ij}]_{v_i} B_{j,l-1}C.
\end{equation}
We prove that the RHS of \eqref{Y1'} equal $0$ in separate cases depending on $c_{ij},i,j\in\I_0$. For $c_{ij}=-1,$ we use the finite type Serre relation \eqref{eq:S1}. For $c_{ij}=-2$, we use the formulas of $\TT_i,\TT_i^{-1}$ in Lemma \ref{lem:Ti}. For $c_{ij}=-3$, we use both of the finite type Serre relation \eqref{eq:S3} and the formulas of $\TT_i,\TT_i^{-1}$. 
\par
We also recall that, by Lemma \ref{lem:bra}(b) and the construction of real root vectors, $\TT_{\omega_j}$ fixes $B_{i,k}$ for any $j\neq i,k\in \Z$ while $\TT_{\omega_i}$ sends $B_{i,k}$ to $o(i)B_{i,k-1}$.
\par
We now start to prove $Y_{1,l}=0$ case by case.\par
(1)$c_{ij}=c_{ji}=0$. In this case, since both $B_{i}, B_{i,1}$ commute with $B_{j,l}$, $\TH_{i,1}$ commutes with $B_{j,l}$ for $l\in \Z$. Hence, $Y_{1,l}=0$. \par
(2)$c_{ij}=c_{ji}=-1$.
We rewrite the finite type Serre relation \eqref{eq:S2} in terms of $q$-brackets as
\begin{equation}
 \big[B_i,[B_i, B_j]_{v_i} \big]_{v_i^{-1}}=-v_i^{-1} B_{j} \K_i,\qquad
 \big[[ B_j, B_i]_{v_i},B_i \big]_{v_i^{-1}}=-v_i^{-1} B_{j} \K_i.
\end{equation}
i.e. each of these two relations is equivalent to $\eqref{eq:S2}$.\par
Applying $o(j)^l\TT_{\omega_j}^{-l}\TT_{\omega_i}^{-k}$ to them, for $k,l\in \Z$, we have
 \begin{equation}\label{Se:brkt}
 \big[B_{i,k},[B_{i,k}, B_{j,l}]_{v_i} \big]_{v_i^{-1}}=-v_i^{-1} B_{j,l} \K_i C^k,\qquad
 \big[[ B_{j,l}, B_{i,k}]_{v_i},B_{i,k} \big]_{v_i^{-1}}=-v_i^{-1} B_{j,l} \K_i C^k.
\end{equation}
We now compute
\begin{align*}
[\TH_{i,1},B_j]\K_i&\overset{\eqref{iDRG3'}}{=}-\big[[B_{i,1},B_i]_{v_i^2},B_j\big]\\
&\overset{\qquad}{=}-\big[B_{i,1},[B_i,B_j]_{v_i }\big]_{v_i}+v_i^2 \big[B_i,[B_{i,1},B_j]_{v_i^{-1}}\big]_{v_i^{-1}}\\
&\overset{\eqref{iDRG2'}}{=} \big[B_{i,1},[ B_{j,-1},B_{i,1}]_{v_i }\big]_{v_i}-v_i^2 \big[B_i,[B_{j,1},B_i]_{v_i^{-1}}\big]_{v_i^{-1}}\\
&\overset{\eqref{Se:brkt}}{=} B_{j,-1} \K_iC -B_{j,1}\K_i.
\end{align*}
Hence, $Y_{1,0}=0$, and by applying $\TT_{\omega_j}^{-l}$, we get $Y_{1,l}=0$.\par
(3)$c_{ij}=-2,c_{ji}=-1$.
We first write $\TT_i(B_j)$ defined in Lemma \ref{lem:Ti} in terms of $q$-brackets as
\begin{align}
[2]_{v_i}\TT_i(B_j) = \big[ [B_j, B_i]_{v_i^2}, B_i\big] + [2]_{v_i} B_j \K_i,\\
[2]_{v_i}\TT^{-1}_i(B_j) = \big[ B_i, [B_i, B_j ]_{v_i^2}\big] + [2]_{v_i} B_j \K_i.
\end{align}
Since $\TT_{ i}, \TT_{\omega_j}$ commute by Lemma \ref{lem:Lus}(a), applying $o(j)^l\TT_{\omega_j}^{-l}$ to these equalities, we have for $l\in \Z$,
\begin{align}\label{Br:brkt1}
[2]_{v_i}\TT_i(B_{j,l}) = \big[ [B_{j,l}, B_i]_{v_i^2}, B_i\big] + [2]_{v_i} B_{j,l} \K_i,\\\label{Br:brkt2}
[2]_{v_i}\TT^{-1}_i(B_{j,l}) = \big[ B_i, [B_i, B_{j,l} ]_{v_i^2}\big] + [2]_{v_i} B_{j,l} \K_i.
\end{align}
Apply $\TT_{\omega_i}^{-1}$ to \eqref{Br:brkt1}, we have
\begin{equation}\label{Br:brkt3}
[2]_{v_i}\TT_{\omega_i}^{-1}\TT_i(B_{j,l}) = \big[ [B_{j,l}, B_{i,1}]_{v_i^2}, B_{i,1}\big] + [2]_{v_i} B_{j,l} \K_i C.
\end{equation}
We now compute
\begin{align*}
[\TH_{i,1},B_j]\K_i&\overset{\eqref{iDRG3'}}{=}-\big[[B_{i,1},B_i]_{v_i^2},B_j\big]\\
&\overset{\qquad}{=}-\big[B_{i,1},[B_i,B_j]_{v_i^2}\big]+v_i^2 \big[B_i,[B_{i,1},B_j]_{v_i^{-2}}\big]\\
&\overset{\eqref{iDRG2'}}{=} \big[B_{i,1},[ B_{j,-1},B_{i,1}]_{v_i^2}\big]-v_i^2 \big[B_i,[B_{j,1},B_{i}]_{v_i^{-2}}\big] \\
&\overset{\eqref{Br:brkt2}}{=}\big[B_{i,1},[ B_{j,-1},B_{i,1}]_{v_i^2}\big] + [2]_{v_i} \TT_i^{-1}(B_{j,1})-[2]_{v_i} B_{j,1}\K_i \\
&\overset{\eqref{Br:brkt3}}{=}-[2]_{v_i} \TT_{\omega_i}^{-1} \TT_i(B_{j,-1})+[2]_{v_i}B_{j,-1}\K_i C + [2]_{v_i} \TT_i^{-1}(B_{j,1})-[2]_{v_i} B_{j,1}\K_i  \\
&\overset{\qquad}{=} [2]_{v_i}B_{j,-1}C\K_i-[2]_{v_i} B_{j,1}\K_i-[2]_{v_i}\big(\TT_{\omega_i}^{-1} \TT_i(B_{j,-1})-\TT_i^{-1}(B_{j,1})\big)\\
&\overset{\qquad}{=}[2]_{v_i}B_{j,-1}C\K_i-[2]_{v_i} B_{j,1}\K_i,
\end{align*}
where the last step follows from $\TT_i \TT_{\omega_i}^{-1} \TT_i= \TT_{\omega_i} \TT_{\omega_j}^{-2}\prod_{k\neq i,j} \TT^{c_{ik}}_{\omega_k}$, which is given by Lemma \ref{lem:Lus}(b).
Hence, $Y_{1,0}=0$, and by applying $\TT_{\omega_j}^{-l}$, we get $Y_{1,l}=0$.\par
(4)$c_{ij}=-3,c_{ji}=-1$. Without loss of generality, assume $o(i)=1$. In this case, we rewrite $\TT_i(B_j)$ in Lemma \ref{lem:Ti} as
\begin{align}
\TT_i(B_j)&=\frac{1}{[3]_{v_i}!}\bigg[\big[[B_j,B_i]_{v_i^3} , B_i\big]_{v_i}, B_i \bigg]_{v_i^{-1}}+\frac{1}{[3]_{v_i}!}v_i^{-1}[B_j,B_i]_{v_i^3} \K_i+ [B_j,B_i]_{v_i}\K_i,\\
\TT_i^{-1}(B_j)&=\frac{1}{[3]_{v_i}!}\bigg[ B_i,\big[ B_i, [B_i, B_j]_{v_i^3}\big]_{v_i} \bigg]_{v_i^{-1}}+\frac{1}{[3]_{v_i}!}v_i^{-1}[ B_i,B_j]_{v_i^3} \K_i+ [B_i,B_j]_{v_i}\K_i.
\end{align}
Since $\TT_i,\TT_{\omega_j}$ commute, applying $o(j)^l\TT_{\omega_j}^{-l}$ for $l\in \Z$ to these equalities, we have
\begin{align}
\TT_i(B_{j,l})&=\frac{1}{[3]_{v_i}!}\bigg[\big[[B_{j,l},B_i]_{v_i^3} , B_i\big]_{v_i}, B_i \bigg]_{v_i^{-1}}+\frac{1}{[3]_{v_i}!}v_i^{-1}[B_{j,l},B_i]_{v_i^3} \K_i+ [B_{j,l},B_i]_{v_i}\K_i,\\
\TT_i^{-1}(B_{j,l})&=\frac{1}{[3]_{v_i}!}\bigg[ B_i,\big[ B_i, [B_i, B_{j,l}]_{v_i^3}\big]_{v_i} \bigg]_{v_i^{-1}}+\frac{1}{[3]_{v_i}!}v_i^{-1}[ B_i,B_{j,l}]_{v_i^3} \K_i+ [B_i,B_{j,l}]_{v_i}\K_i.
\end{align}
In particular, for $l=-1$ and $1$ respectively, we have
\begin{align}
\label{Br:brkt4}[\TT_i(B_{j,-1}),B_i]_{v_i^{-3}}&=\frac{1}{[3]_{v_i}!}\bigg[\bigg[\big[[B_{j,-1},B_i]_{v_i^3} , B_i\big]_{v_i}, B_i \bigg]_{v_i^{-1}},B_i\bigg]_{v_i^{-3}}\\\notag
&+\frac{1}{[3]_{v_i}!}v_i^{-1}\big[[B_{j,-1},B_i]_{v_i^3},B_i\big]_{v_i^{-3}} \K_i+ \big[[B_{j,-1},B_i]_{v_i}, B_i\big]_{v_i^{-3}}\K_i,\\
\label{Br:brkt5}[B_i,\TT_i^{-1}(B_{j,1})]_{v_i^{-3}}&=\frac{1}{[3]_{v_i}!}\bigg[B_i,\bigg[ B_i,\big[ B_i, [B_i, B_{j,1}]_{v_i^3}\big]_{v_i} \bigg]_{v_i^{-1}}\bigg]_{v_i^{-3}}\\\notag
&+\frac{1}{[3]_{v_i}!}v_i^{-1}\big[ B_i, [ B_i,B_{j,1}]_{v_i^3}\big]_{v_i^{-3}} \K_i+\big[B_i, [B_i,B_{j,1}]_{v_i}\big]\K_i.
\end{align}
Apply $\TT_{\omega_i}^{-1}$ to \eqref{Br:brkt4}, since $\TT_{\omega_i}^{-1}\TT_i(B_i)=\TT_{\omega_i}^{-1} (B_i \K_i^{-1})=B_{i,1}\K_i^{-1}C^{-1}$, we have
\begin{align}
\label{Br:brkt6}\TT_{\omega_i}^{-1}\TT_i\big([B_{j,-1},B_{i }]_{v_i^{-3}}\big)\K_iC 
=&\frac{1}{[3]_{v_i}!}\bigg[\bigg[\big[[B_{j,-1},B_{i,1}]_{v_i^3} , B_{i,1}]\big]_{v_i}, B_{i,1} \bigg]_{v_i^{-1}},B_{i,1}\bigg]_{v_i^{-3}}\\\notag
 &+\frac{1}{[3]_{v_i}!}v_i^{-1}\big[[B_{j,-1},B_{i,1}]_{v_i^3},B_{i,1}\big]_{v_i^{-3}} \K_i C\\\notag
 &+ \big[[B_{j,-1},B_{i,1}]_{v_i}, B_{i,1}\big]_{v_i^{-3}}\K_iC.
\end{align}
On the other hand, we also rewrite the finite type Serre relation \eqref{eq:S4} as
\begin{align}\label{Se:brkt2}
\bigg[\bigg[\big[[B_{j},B_i]_{v_i^3} , B_i\big]_{v_i}, B_i \bigg]_{v_i^{-1}},B_i\bigg]_{v_i^{-3}}&=-v_i^{-1}(1+[3]_{v_i}^2)( B_{j} B_{i}^2+ B_{i}^2 B_{j})\K_i \\\notag
& +v_i^{-1}[4]_{v_i} (1+[2]_{v_i}^2) B_{i} B_{j} B_{i} \K_i -v_i^{-2}[3]^2_{v_i} B_{j} \K_i^2.
\end{align}
Apply $o(j)\TT^{-1}_{\omega_j}$ and we obtain
\begin{align}\label{Se:brkt3}
\bigg[\bigg[\big[[B_{j,1},B_i]_{v_i^3} , B_i\big]_{v_i}, B_i \bigg]_{v_i^{-1}},B_i\bigg]_{v_i^{-3}}&=-v_i^{-1}(1+[3]_{v_i}^2)( B_{j,1} B_{i}^2+ B_{i}^2 B_{j,1})\K_i \\\notag
& +v_i^{-1}[4]_{v_i} (1+[2]_{v_i}^2) B_{i} B_{j,1} B_{i} \K_i -v_i^{-2}[3]^2_{v_i} B_{j,1} \K_i^2.
\end{align}
Note the leading term (of degree $5$) on the RHS of \eqref{Br:brkt5} coincides with the LHS of \eqref{Se:brkt3}. We substitute it using  \eqref{Se:brkt3} and simplify as
\begin{align}\label{Se:br1}
[3]_{v_i}[2]_{v_i}[B_i,\TT_i^{-1}(B_{j,1})]_{v_i^{-3}}&=[3]_{v_i} \big[B_i,[B_{j,1},B_i]_{v_i^{-3}}\big]_{v_i} \K_i -v_i^{-2}[3]^2_{v_i} B_{j,1} \K_i^2.
\end{align}
Similarly, we apply $\TT_{\omega_j} \TT_{\omega_i}^{-1}$ to \eqref{Se:brkt2}, by Lemma \ref{lem:jac}(a),
\begin{align}\notag
\bigg[B_{i,1},\bigg[B_{i,1} ,\big[ B_{i,1},&[B_{i,1},B_{j,-1}]_{v_i^3}\big]_{v_i} \bigg]_{v_i^{-1}}\bigg]_{v_i^{-3}}=-v_i^{-1}(1+[3]_{v_i}^2)( B_{j,-1} B_{i,1}^2+ B_{i,1}^2 B_{j,-1})\K_i \\\label{Se:brkt4}
& +v_i^{-1}[4]_{v_i} (1+[2]_{v_i}^2) B_{i,1} B_{j,-1} B_{i,1} \K_i C-v_i^{-2}[3]^2_{v_i} B_{j,-1} \K_i^2 C^2.
\end{align}
and then we substitute the leading term of RHS of \eqref{Br:brkt6} using \eqref{Se:brkt4}. We obtain
\begin{align*}
[3]_{v_i}[2]_{v_i}\TT_{\omega_i}^{-1}\TT_i [B_{j,-1} ,B_{i}]_{v_i^{-3}} \K_i C&=[3]_{v_i} \big[[B_{i,1},B_{j,-1}]_{v_i^{-3}},B_{i,1}\big]_{v_i} \K_i C -v_i^{-2}[3]^2_{v_i} B_{j,-1} \K_i^2 C^2,
\end{align*}
which can be simplified as
\begin{align}\label{Se:br2}
 [2]_{v_i}\TT_{\omega_i}^{-1}\TT_i [B_{j,-1} ,B_{i}]_{v_i^{-3}} &= \big[[B_{i,1},B_{j,-1}]_{v_i^{-3}},B_{i,1}\big]_{v_i}  -v_i^{-2}[3]_{v_i} B_{j,-1} \K_i C.
\end{align}
We now compute $Y_{1,0}$ in this case.
\begin{align*}
[\TH_{i,1},B_j]\K_i&\overset{\eqref{iDRG3'}}{=}-\big[[B_{i,1},B_i]_{v_i^2},B_j\big]\\
&\overset{\qquad}{=}-\big[B_{i,1},[B_i,B_j]_{v_i^3 }\big]_{v_i^{-1}}+v_i^2 \big[B_i,[B_{i,1},B_j]_{v_i^{-3}}\big]_{v_i }\\
&\overset{\eqref{iDRG2'}}{=} \big[B_{i,1},[ B_{j,-1},B_{i,1}]_{v_i^3 }\big]_{v_i^{-1}}-v_i^2 \big[B_i,[B_{j,1},B_i]_{v_i^{-3}}\big]_{v_i}\\
&\overset{\qquad }{=}v_i^2 \big[[B_{i,1},B_{j,-1}]_{v_i^{-3}},B_{i,1}\big]_{v_i}-v_i^2 \big[B_i,[B_{j,1},B_i]_{v_i^{-3}}\big]_{v_i}\\
&\overset{\ \,(*)\,\ }{=}[3]_{v_i} B_{j,-1} \K_iC -[3]_{v_i}B_{j,1}\K_i\\
&\qquad +v_i^2[2]_{v_i}\bigg(\TT_{\omega_i}^{-1}\TT_i [B_{j,-1} ,B_{i}]_{v_i^{-3}} -[B_i,\TT_i^{-1}(B_{j,1})]_{v_i^{-3}}\K_i^{-1}\bigg)\\
&\overset{\qquad}{=}[3]_{v_i} B_{j,-1} \K_iC -[3]_{v_i}B_{j,1}\K_i+v_i^2[2]_{v_i}\TT_i^{-1}\bigg(\TT_i \TT_{\omega_i}^{-1}\TT_i [B_{j,-1} ,B_{i}]_{v_i^{-3}} -[B_i, B_{j,1} ]_{v_i^{-3}}\bigg)\\
&\overset{\ (**)\ }{=}[3]_{v_i} B_{j,-1} \K_iC -[3]_{v_i}B_{j,1}\K_i+v_i^2[2]_{v_i}\TT_i^{-1}\bigg({\color{red}-}[B_{j,2} ,B_{i,-1}]_{v_i^{-3}} -[B_i, B_{j,1} ]_{v_i^{-3}} \bigg),\\
&\overset{\eqref{iDRG2'}}{=}[3]_{v_i} B_{j,-1} \K_iC -[3]_{v_i}B_{j,1}\K_i
\end{align*}
where step (*) follows by applying \eqref{Se:br2} to the first term and applying \eqref{Se:br1} to the second term, and step (**) follows from $\TT_i \TT_{\omega_i}^{-1} \TT_i= \TT_{\omega_i} \TT_{\omega_j}^{-3} $ given in Lemma \ref{lem:Lus}(b). (also $o(j)=-1$ gives the red additional sign in this step) Hence, $Y_{1,0}=0$ and by applying $\TT_{\omega_j}^{-l}$ we have  $Y_{1,l}=0$ for $l\in \Z.$

\section{Verification of Serre relations}\label{Serre}
The goal of this section is to establish general Serre relations \eqref{iDRG5}-\eqref{iDRG6} in $\tUiD_{red}$. We first recover the general Serre relation \eqref{iDRG5} formulated in \cite{LW20b} for $c_{ij}=-1$, using a more straightforward approach compared with the original one. We generalize this approach and offer several formulations for the Serre relation for $c_{ij}=-2$. We obtain two symmetric formulations in \S \ref{symfun}. Using these symmetric formulations, we derive the relation \eqref{iDRG6} in terms of generating functions and finish the proof of Theorem \ref{DprBCF} in \S \ref{genfor}.

\subsection{Serre relation for $c_{ij}=-1$}\label{SeADE}
Let $c_{ij}=-1$. Recall the notation $\bS(k_1,k_2|l;i,j)$ introduced in \eqref{eq:Skk} and denote it by $\bS(k_1,k_2|l)$ for short. The Serre relation \eqref{iDRG4'}, together with the relation~ \eqref{iDRG1'}, is verified in \cite[\S 4.7-4.8]{LW20b}, using a spiral induction. In this section, we give a new proof for the Serre relation \eqref{iDRG4'}, without the help of relation \eqref{iDRG1'}. To begin with, we recall two technical lemmas from their paper.

\begin{lemma}[\text{\cite[Lemma 4.13]{LW20b}}]\label{lem:SSS}
For $k_1, k_2, l \in \Z$, we have
\begin{align*}
& \bS(k_1,k_2+1 |l) + \bS(k_1+1,k_2|l) -[2]_{v_i} \bS(k_1,k_2|l+1)
 \\
&=\Sym_{k_1,k_2} \Big( -[\Theta_{i, k_2-k_1+1}, B_{jl}]_{v_i^{-2}} C^{k_1} +v_i^{-2} [\Theta_{i, k_2-k_1-1}, B_{jl}]_{v_i^{-2}} C^{k_1+1} \Big) \K_i  .
\end{align*}
\end{lemma}
\begin{lemma}[\text{\cite[Lemma 4.9]{LW20b}}] \label{lem:SS}
For $k_1, k_2, l \in \Z$, we have
\begin{align*}
& \bS(k_1,k_2+1 |l) + \bS(k_1+1,k_2|l) -[2]_{v_i} \bS(k_1+1,k_2+1|l-1)
 \\
&= \Sym_{k_1,k_2}\Big( -[B_{jl},\Theta_{i, k_2-k_1+1} ]_{v_i^{-2}} C^{k_1} +v_i^{-2} [B_{jl},\Theta_{i, k_2-k_1-1}]_{v_i^{-2}} C^{k_1+1} \Big) \K_i .
\end{align*}
\end{lemma}
Denote
\begin{align*}
\bS(w_1,w_2|z)=\Sym_{w_1,w_2}\big\{\B_{i}(w_1)\B_{i}(w_2)\B_{j}(z)
-[2]_{v_i}\B_{i}(w_1)\B_{j}(z)\B_{i}(w_2)+\B_{j}(z)\B_{i}(w_1)\B_{i}(w_2)\big\}.
\end{align*}
Lemma \ref{lem:SSS} and \ref{lem:SS} can be written in terms of generating functions respectively as
\begin{align}
\label{gen7}&(w_1^{-1}+w_2^{-1}-[2]_{v_i}z^{-1})\bS(w_1,w_2|z)
\\\notag
=&\Sym_{w_1,w_2} \frac{\bDel(w_1 w_2)}{v_i-v_i^{-1}} (v_i^{-2} w_1^{-1}-w_2^{-1})[\bTH_i(w_2),\B_{j}(z)]_{v_i^{-2}}\K_i,
\end{align}
and
\begin{align}
\label{gen6}&(w_1+w_2-[2]_{v_i}z)\bS(w_1,w_2|z)
\\\notag
=&\Sym_{w_1,w_2} \frac{\bDel(w_1w_2)}{v_i-v_i^{-1}} (v_i^{-2} w_2-w_1)[\B_{j}(z),\bTH_i(w_2)]_{v_i^{-2}}\K_i.
\end{align}
Then we calculate $\eqref{gen7}\times [2]_{v_i}z+\eqref{gen6} \times (w_1^{-1}+w_2^{-1})$ and we obtain
\begin{align}\notag
&(w_1-v_i^2 w_2)(w_2^{-1}-v_i^{-2}w_1^{-1})\bS(w_1,w_2|z)
\\\label{gen8}
=&[2]_{v_i} z\Sym_{w_1,w_2} \frac{\bDel(w_1w_2)}{v_i-v_i^{-1}}(v_i^{-2} w_1^{-1}-w_2^{-1})[\bTH_i(w_2),\B_{j}(z)]_{v_i^{-2}}\K_i
\\\notag
&+\Sym_{w_1,w_2} \frac{\bDel(w_1w_2)}{v_i-v_i^{-1}} (v_i^{-2} w_2-w_1)(w_1^{-1}+w_2^{-1})[\B_{j}(z),\bTH_i(w_2)]_{v_i^{-2}}\K_i.
\end{align}
We also calculate $\eqref{gen7} \times (w_1+w_2)+\eqref{gen6}\times [2]_{v_i} z^{-1}$ and we obtain
\begin{align}\notag
&(w_1-v_i^2 w_2)(w_2^{-1}-v_i^{-2}w_1^{-1})\bS(w_1,w_2|z)
\\\label{gen8'}
=&\Sym_{w_1,w_2} \frac{\bDel(w_1w_2)}{v_i-v_i^{-1}}(v_i^{-2} w_1^{-1}-w_2^{-1})(w_1+w_2)[\bTH_i(w_2),\B_{j}(z)]_{v_i^{-2}}\K_i
\\\notag
&+[2]_{v_i} z^{-1}\Sym_{w_1,w_2} \frac{\bDel(w_1w_2)}{v_i-v_i^{-1}} (v_i^{-2} w_2-w_1)[\B_{j}(z),\bTH_i(w_2)]_{v_i^{-2}}\K_i.
\end{align}
Simplify \eqref{gen8} as
\begin{align}
\bS(w_1,w_2|z) \label{gen9}
=&-\Sym_{w_1,w_2} \frac{\bDel(w_1w_2)}{v_i-v_i^{-1}}\frac{[2]_{v_i} z}{w_1-v_i^2 w_2}[\bTH_i(w_2),\B_{j}(z)]_{v_i^{-2}}\K_i
\\\notag
&-\Sym_{w_1,w_2} \frac{\bDel(w_1w_2)}{v_i-v_i^{-1}} \frac{w_1+w_2}{w_1-v_i^2 w_2}[\B_{j}(z),\bTH_i(w_2)]_{v_i^{-2}}\K_i,
\end{align}
which is exactly \eqref{iDRG5}.

\subsection{Symmetric formulation}\label{symfun}
We now forward to the case $c_{ij}=-2$. Two symmetric formulations \eqref{ind6} and \eqref{ind7}, generalizing Lemma \ref{lem:SSS} and \ref{lem:SS}, are formulated and verified in this section. Denote
\begin{align*}
S(k_1,k_2,k_3|l)=\Sym_{k_1,k_2,k_3}\sum_{s=0}^3(-1)^s
\qbinom{3}{s}_{v_i}
B_{i,k_1}\cdots B_{i,k_s}B_{j,l} B_{i,k_s+1} \cdots B_{i,k_3}.
\end{align*}
Note that $S(k_1,k_2,k_3|l)$ is symmetric with respect to the first three components.
\begin{proposition}\label{symS}
We have, for any $k_1,k_2,k_3$,
\begin{align}
\label{ind6}&S(k_1,k_2,k_3+1|l)+S(k_1+1,k_2,k_3|l)+S(k_1,k_2+1,k_3|l)-[3]_{v_i}S(k_1,k_2,k_3|l+1)\\\notag
=&\frac{1}{2}\Sym_{k_1,k_2,k_3}\Big(-v_i^{-1}[2]_{v_i}\big[[\Theta_i(k_2,k_3),B_{j,l}]_{v_i^{-2}},B_{i,k_1}\big]_{v_i^{2}}+\big[\Theta_i(k_2,k_3),[B_{i,k_1},B_{j,l}]_{v_i^2}\big]_{v_i^{-4}}\Big)
\end{align}
\end{proposition}
The following relation can be obtained from \eqref{ind6} by applying $\Psi$.
\begin{align}\notag
&S(k_1-1,k_2,k_3|l)+S(k_1 ,k_2-1,k_3 |l)+S(k_1 ,k_2 ,k_3-1|l)-[3]_{v_i} S(k_1 ,k_2,k_3|l-1)\\
\label{ind7}=&\frac{1}{2}v_i^{-1}[2]_{v_i}\Sym_{k_1,k_2,k_3}\big[B_{i,k_3},[B_{j,l},\Theta_i(k_1-1,k_2-1)]_{v_i^{-2}}\big]_{v_i^2} \\\notag
&- \frac{1}{2}\Sym_{k_1,k_2,k_3} \big[[B_{j,l},B_{i,k_3 }]_{v_i^2},\Theta_i(k_1-1,k_2-1)\big]_{v_i^{-4}}
\\\notag
\end{align}
\begin{proof}[Proof of Proposition \ref{symS}]
Recall the definition of $\Theta_i(s,r)$ from \eqref{eq:Th1}. Since $\Theta_i(s,r)=\Theta_i(r,s)$, we have $\Sym_{r,s} \Theta_i(s,r)=2\Theta_i(s,r)$. We rewrite \eqref{ind6} in the following equivalent form,
\begin{align}\notag
&S(k_1,k_2,k_3+1|l)+S(k_1+1,k_2,k_3|l)+S(k_1,k_2+1,k_3|l)-[3]_{v_i}S(k_1,k_2,k_3|l+1)\\\label{ind5-1}
=&-v_i^{-1}[2]_{v_i}\big[[\Theta_i(k_2,k_3),B_{j,l}]_{v_i^{-2}},B_{i,k_1}\big]_{v_i^{2}}+\big[\Theta_i(k_2,k_3),[B_{i,k_1},B_{j,l}]_{v_i^2}\big]_{v_i^{-4}}
\\\notag
&-v_i^{-1}[2]_{v_i}\big[[\Theta_i(k_1,k_3),B_{j,l}]_{v_i^{-2}},B_{i,k_2}\big]_{v_i^{2}}+\big[\Theta_i(k_1,k_3),[B_{i,k_2},B_{j,l}]_{v_i^2}\big]_{v_i^{-4}}
\\\notag
&-v_i^{-1}[2]_{v_i}\big[[\Theta_i(k_1,k_2),B_{j,l}]_{v_i^{-2}},B_{i,k_3}\big]_{v_i^{2}}+\big[\Theta_i(k_1,k_2),[B_{i,k_3},B_{j,l}]_{v_i^2}\big]_{v_i^{-4}}.
\end{align}
Denote
\begin{align}\notag
R(k_1,k_2,k_3|l)=\Sym_{k_1,k_2} \bigg(B_{i,k_1}B_{i,k_2}[B_{i,k_3},B_{j,l}]_{v_i^{2}}
&-v_i^{-1}[2]_{v_i}B_{i,k_1}[B_{i,k_3},B_{j,l}]_{v_i^{2}}B_{i,k_2}\\
&\quad+v_i^{-2}[B_{i,k_3},B_{j,l}]_{v_i^{2}}B_{i,k_1}B_{i,k_2}\bigg).
\end{align}
Note that $R(k_1,k_2,k_3|l)$ is only symmetric with respect to its first two components. In fact, $R(k_1,k_2,k_3|l)$ plays the role of breaking the symmetry of $S(k_1,k_2,k_3|l)$ as it satisfies
\begin{equation}\label{ind0}
S(k_1,k_2,k_3|l)=R(k_1,k_2,k_3|l)+R(k_1,k_3,k_2|l)+R(k_2,k_3,k_1|l).
\end{equation}
We compute
\begin{align}\notag
&S(k_1,k_2,k_3+1|l)-[3]_{v_i} R(k_1,k_2,k_3|l+1)\\\label{ind1}
=&\bigg\{(1+v_i^2)B_{i,k_1}[B_{i,k_3+1},B_{i,k_2}]_{v_i^2}B_{j,l}+[B_{i,k_3+1},B_{i,k_1}]_{v^2_i}B_{i,k_2}B_{j,l}\\\notag
&-[3]_{v_i}[B_{i,k_3+1},B_{i,k_1}]_{v_i^2}B_{j,l} B_{i,k_2}-v_i^{-2}[3]_{v_i}B_{i,k_1} B_{j,l} [B_{i,k_3+1},B_{i,k_2}]_{v_i^2}\\\notag
&+v_i^{-2}B_{j,l} B_{i,k_1} [B_{i,k_3+1},B_{i,k_2}]_{v_i^2}+(v_i^{-2}+v_i^{-4})B_{j,l}[B_{i,k_3+1},B_{i,k_1}]_{v_i^2}B_{i,k_2}\bigg\}\\\notag
&+\{k_1\leftrightarrow k_2\}.
\end{align}
where $\{k_1\leftrightarrow k_2\}$ represents the element obtained by swapping $k_1,k_2$ in the first curly brackets.\par
Rewrite \eqref{ind1} using the symmetrizer as
\begin{align}\notag
&S(k_1,k_2,k_3+1|l)-[3]_{v_i} R(k_1,k_2,k_3|l+1)\\\label{ind2}
=&\bigg\{(1+v_i^2)B_{i,k_1}[B_{i,k_3+1},B_{i,k_2}]_{v_i^2}B_{j,l}+[B_{i,k_3+1},B_{i,k_2}]_{v^2_i}B_{i,k_1}B_{j,l}\\\notag
&-[3]_{v_i}[B_{i,k_3+1},B_{i,k_2}]_{v_i^2}B_{j,l}B_{i,k_1}-v_i^{-2}[3]_{v_i}B_{i,k_1} B_{j,l} [B_{i,k_3+1},B_{i,k_2}]_{v_i^2}\\\notag
&+v_i^{-2}B_{j,l} B_{i,k_1} [B_{i,k_3+1},B_{i,k_2}]_{v_i^2}+(v_i^{-2}+v_i^{-4})B_{j,l}[B_{i,k_3+1},B_{i,k_2}]_{v_i^2}B_{i,k_1}\bigg\}\\\notag
&+\{k_1\leftrightarrow k_2\}.
\end{align}
On the other hand, the relation \eqref{iDRG3'} implies that
\begin{align}\label{ind21}
[B_{i,k_3+1},B_{i,k_1}]_{v_i^2}&=\Theta_i(k_1,k_3)-[B_{i,k_1+1},B_{i,k_3}]_{v_i^2},\\\label{ind22}
[B_{i,k_3+1},B_{i,k_2}]_{v_i^2}&=\Theta_i(k_2,k_3)-[B_{i,k_2+1},B_{i,k_3}]_{v_i^2}.
\end{align}

Substitute \eqref{ind21} and \eqref{ind22} into \eqref{ind2}, and we have
\begin{align}\notag
&S(k_1,k_2,k_3+1|l)-[3]_{v_i} R(k_1,k_2,k_3|l+1)\\\label{ind3}
=&-\bigg\{(1+v_i^2)B_{i,k_1}[B_{i,k_2+1},B_{i,k_3}]_{v_i^2}B_{j,l}+[B_{i,k_2+1},B_{i,k_3}]_{v_i^2}B_{i,k_1}B_{j,l}\\\notag
&-[3]_{v_i}[B_{i,k_2+1},B_{i,k_3}]_{v_i^2}B_{j,l}B_{i,k_1}-v_i^{-2}[3]_{v_i}B_{i,k_1} B_{j,l} [B_{i,k_2+1},B_{i,k_3}]_{v_i^2}\\\notag
&+v_i^{-2}B_{j,l} B_{i,k_1} [B_{i,k_2+1},B_{i,k_3}]_{v_i^2}+(v_i^{-2}+v_i^{-4})B_{j,l}[B_{i,k_2+1},B_{i,k_3}]_{v_i^2}B_{i,k_1}\bigg\}\\\notag
&-\{k_1\leftrightarrow k_2\}\\\notag
&+Q_{1,2},
\end{align}
where $Q_{1,2}$ denotes all terms involving the imaginary root vectors
\begin{align*}
Q_{1,2}=&-v_i^{-1}[2]_{v_i}\big[[\Theta_i(k_2,k_3),B_{j,l}]_{v_i^{-2}},B_{i,k_1}\big]_{v_i^{2}}+\big[\Theta_i(k_2,k_3),[B_{i,k_1},B_{j,l}]_{v_i^2}\big]_{v_i^{-4}}
\\
&-v_i^{-1}[2]_{v_i}\big[[\Theta_i(k_1,k_3),B_{j,l}]_{v_i^{-2}},B_{i,k_2}\big]_{v_i^{2}}+\big[\Theta_i(k_1,k_3),[B_{i,k_2},B_{j,l}]_{v_i^2}\big]_{v_i^{-4}}.
\end{align*}
We recognize that the first $\bigg\{\ \bigg\}$-term on the RHS of \eqref{ind3} is the same as the $\bigg\{\ \bigg\}$-term on the RHS of \eqref{ind2}, up to a swap of indices $k_2\leftrightarrow k_3$. Hence, we replace the one in \eqref{ind3} by \eqref{ind2}. We do the same thing for the term $\{k_1\leftrightarrow k_2\}$ in \eqref{ind3}.
\begin{align}\notag
&S(k_1,k_2,k_3+1|l)-[3]_{v_i}R(k_1,k_2,k_3|l+1)\\\notag
=&-\bigg(S(k_1,k_3,k_2+1|l)-[3]_{v_i} R(k_1,k_3,k_2|l+1)\bigg)\\\notag
&+\bigg\{(1+v_i^2)B_{i,k_3}[B_{i,k_2+1},B_{i,k_1}]_{v_i^2}B_{j,l}+[B_{i,k_2+1},B_{i,k_1}]_{v_i^2}B_{i,k_3}B_{j,l}\\\label{ind4}
&-[3]_{v_i}[B_{i,k_2+1},B_{i,k_1}]_{v_i^2}B_{j,l}B_{i,k_3}-v_i^{-2}[3]_{v_i}B_{i,k_3} B_{j,l} [B_{i,k_2+1},B_{i,k_1}]_{v_i^2}\\\notag
&+v_i^{-2}B_{j,l} B_{i,k_3} [B_{i,k_2+1},B_{i,k_1}]_{v_i^2}+(v_i^{-2}+v_i^{-4})B_{j,l}[B_{i,k_2+1},B_{i,k_1}]_{v_i^2}B_{i,k_3}\bigg\}\\\notag
&-\bigg(S(k_2,k_3,k_1+1|l)-[3]_{v_i} R(k_2,k_3,k_1|l+1)\bigg)+\bigg\{k_1\leftrightarrow k_2\bigg\}\\\notag
&+Q_{1,2}.
\end{align}
By \eqref{iDRG3'}, we have the following relation
\begin{equation}
[B_{i,k_2+1},B_{i,k_1}]_{v_i^2}+[B_{i,k_1+1},B_{i,k_2}]_{v_i^2}=\Theta_i(k_1,k_2).
\end{equation}
Now we can apply the above relation to those two $\bigg\{\ \bigg\}$-terms in \eqref{ind4} and we obtain
\begin{align}\notag
&S(k_1,k_2,k_3+1|l)+S(k_1+1,k_2,k_3|l)+S(k_1,k_2+1,k_3|l)
\\\notag
&-[3]_{v_i}\big( R(k_1,k_2,k_3|l+1)+  R(k_1,k_3,k_2|l+1)+  R(k_2,k_3,k_1|l+1)\big)
\\
\label{ind5}=&-v_i^{-1}[2]_{v_i}\big[[\Theta_i(k_2,k_3),B_{j,l}]_{v_i^{-2}},B_{i,k_1}\big]_{v_i^{2}}+\big[\Theta_i(k_2,k_3),[B_{i,k_1},B_{j,l}]_{v_i^2}\big]_{v_i^{-4}}
\\\notag
&-v_i^{-1}[2]_{v_i}\big[[\Theta_i(k_1,k_3),B_{j,l}]_{v_i^{-2}},B_{i,k_2}\big]_{v_i^{2}}+\big[\Theta_i(k_1,k_3),[B_{i,k_2},B_{j,l}]_{v_i^2}\big]_{v_i^{-4}}
\\\notag
&-v_i^{-1}[2]_{v_i}\big[[\Theta_i(k_1,k_2),B_{j,l}]_{v_i^{-2}},B_{i,k_3}\big]_{v_i^{2}}+\big[\Theta_i(k_1,k_2),[B_{i,k_3},B_{j,l}]_{v_i^2}\big]_{v_i^{-4}}.
\end{align}
Finally, by \eqref{ind0}, we obtain \eqref{ind5-1} from \eqref{ind5}, as desired. 

\end{proof}

\subsection{Generating function formulation}\label{genfor}
By taking suitable linear combination of two symmetric formulations, we derive the Serre relation \eqref{iDRG6} and thus finish the proof of Theorem \ref{DprBCF}.\par
Fix $i,j\in \I_0$ such that $c_{ij}=-2$. Recall the notation $\mathbb{S}(w_1,w_2,w_3|z;i,j)$ from \eqref{eq:BCF} and denote it by $\mathbb{S}(w_1,w_2,w_3|z)$ for short.\par
We can rewrite \eqref{ind6} in terms of generating functions as
\begin{align}\notag
&(w_1^{-1}+w_2^{-1}+w_3^{-1}-[3]_{v_i}z^{-1})\mathbb{S}(w_1,w_2,w_3|z)\\
\label{gen1}=&-v_i[2]_{v_i}\Sym_{w_1,w_2,w_3} \frac{\bDel(w_2w_3)}{v_i-v_i^{-1}}(v_i^{-2}w_3^{-1}-w_2^{-1})\big[[\bTH_i(w_2),\B_j(z)]_{v_i^{-4}}, \B_i(w_1)\big]\K_i\\\notag
&+\Sym_{w_1,w_2,w_3} \frac{\bDel(w_2w_3)}{v_i-v_i^{-1}}(v_i^{-2}w_3^{-1}-w_2^{-1})\big[\bTH_i(w_2),[\B_i(w_1), \B_j(z)]_{v_i^{-2}}\big]\K_i,
\end{align}
 and rewrite \eqref{ind7} in terms of generating functions as
 \begin{align}\notag
&(w_1+w_2+w_3-[3]_{v_i}z)\bS(w_1,w_2,w_3|z)\\
\label{gen2}=&v_i[2]_{v_i}\Sym_{w_1,w_2,w_3} \frac{\bDel(w_2w_3)}{v_i-v_i^{-1}}(v_i^{-2}w_2-w_3)\big[\B_i(w_1),[\B_j(z),\bTH_i(w_2)]_{v_i^{-4}}\big]\K_i\\\notag
&-\Sym_{w_1,w_2,w_3} \frac{\bDel(w_2w_3)}{v_i-v_i^{-1}}(v_i^{-2}w_2-w_3)\big[[ \B_j(z),\B_i(w_1)]_{v_i^{-2}},\bTH_i(w_2)\big]\K_i.
\end{align}
We calculate $\eqref{gen2}\times [3]_{v_i}z^{-1}+\eqref{gen1}\times (w_1+w_2+w_3)$ and obtain
\begin{align}
\label{gen4'}&\big((w_1+w_2+w_3)(w_1^{-1}+w_2^{-1}+w_3^{-1})-[3]^2_{v_i}\big)\mathbb{S}(w_1,w_2,w_3|z)\K_i^{-1}\\\notag
=&[3]_{v_i}z^{-1}\Sym_{w_1,w_2,w_3} \frac{\bDel(w_2w_3)}{v_i-v_i^{-1}}(v_i^{-2}w_2-w_3)\\\notag
&\quad\times\bigg(v_i[2]_{v_i}\big[\B_i(w_1),[\B_j(z),\bTH_i(w_2)]_{v_i^{-4}}\big]-\big[[ \B_j(z),\B_i(w_1)]_{v_i^{-2}},\bTH_i(w_2)\big]\bigg)\\\notag
&+\Sym_{w_1,w_2,w_3} \frac{\bDel(w_2w_3)}{v_i-v_i^{-1}}(v_i^{-2}w_3^{-1}-w_2^{-1})(w_1+w_2+w_3)\\\notag
&\quad\times\bigg(-v_i[2]_{v_i}\big[[\bTH_i(w_2),\B_j(z)]_{v_i^{-4}}, \B_i(w_1)\big]+\big[\bTH_i(w_2),[\B_i(w_1), \B_j(z)]_{v_i^{-2}}\big]\bigg).
\end{align}
Dividing both sides of \eqref{gen4'} by the coefficient of $\mathbb{S}(w_1,w_2,w_3|z)$, we obtain the defining relation \eqref{iDRG6} of $\tUiD$.

\begin{remark}\label{double} Our formulation \eqref{iDRG6} of the Serre relation for $c_{ij}=-2$ is not unique since an alternative formulation can be derived as follows. We calculate $\eqref{gen2}\times(w_1^{-1}+w_2^{-1}+w_3^{-1}) + \eqref{gen1}\times [3]_{v_i}z$ and obtain a variant of \eqref{gen4'} as
\begin{align}
\label{gen4}&\big((w_1+w_2+w_3)(w_1^{-1}+w_2^{-1}+w_3^{-1})-[3]^2_{v_i}\big)\bS(w_1,w_2,w_3|z)\K_i^{-1}\\\notag
=&\Sym_{w_1,w_2,w_3} \frac{\bDel(w_2w_3)}{v_i-v_i^{-1}}(v_i^{-2}w_2-w_3)(w_1^{-1}+w_2^{-1}+w_3^{-1})\\\notag
&\quad\times \bigg(v_i[2]_{v_i}\big[\B_i(w_1),[\B_j(z),\bTH_i(w_2)]_{v_i^{-4}}\big]- \big[[ \B_j(z),\B_i(w_1)]_{v_i^{-2}},\bTH_i(w_2)\big]\bigg)\\\notag
&+[3]_{v_i}z\Sym_{w_1,w_2,w_3} \frac{\bDel(w_2w_3)}{v_i-v_i^{-1}}(v_i^{-2}w_3^{-1}-w_2^{-1})\\\notag
&\quad\times\bigg(-v_i[2]_{v_i}\big[[\bTH_i(w_2),\B_j(z)]_{v_i^{-4}},\B_i(w_1)\big]+\big[\bTH_i(w_2),[\B_i(w_1), \B_j(z)]_{v_i^{-2}}\big]\bigg).
\end{align}
Dividing both sides of \eqref{gen4} by the coefficient of $\bS(w_1,w_2,w_3|z)$, we get an alternative formulation for the Serre relation, which looks different from \eqref{iDRG6}. In fact, \eqref{gen4} can be obtained from \eqref{gen4'} by applying $\Psi$, and thus the alternative formulation of the Serre relation can also be obtained from \eqref{iDRG6} by applying $\Psi$.
\end{remark}

\end{document}